\RequirePackage{amsthm,amsmath,mathrsfs,amsfonts,amssymb}
\documentclass[12pt,a4paper]{article}
\pdfoutput=1

\usepackage{lmodern}
\usepackage{geometry}
\usepackage{setspace}
\usepackage{float}

\usepackage{tikz-cd}
\tikzset{%
    symbol/.style={%
        draw=none,
        every to/.append style={%
            edge node={node [sloped, allow upside down, auto=false]{$#1$}}}
    }
}
\tikzcdset{
  every label/.append style={fill=none, draw=none, inner sep=1pt}
}
\usetikzlibrary{patterns}
\usetikzlibrary{backgrounds}
\usepackage{mathtools}
\usepackage{mathabx} 
\usepackage{enumerate}
\usepackage{arrayjobx}
\usepackage{lmodern}
\usepackage{microtype} 

\newcommand{\Rel}{\mbox{\bf Rel}}
\newcommand{\RelMon}{\mbox{\bf RelMon}}
\newcommand{\RelFrob}{\mbox{\bf RelFrob}}
\newcommand{\Hom}{\mbox{\rm Hom}}

\newcommand{\eSSet}{\mathrm{Set}^\epsilon_\Delta}
\newcommand{\SSet}{\mathrm{Set}_\Delta}
\newcommand{\eHorn}{\epsilon-\mathrm{Horn}}
\newcommand*\Sim{\includegraphics[scale=0.26]{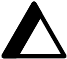}}
\newcommand{\eps}{{\bf \epsilon}}
\DeclareRobustCommand{\ELambda}{\text{\reflectbox{$\bf\Lambda$}}}
\def\dashmapsto{\mapstochar\dashrightarrow}

\tikzset{LA/.style = {
                      line width=#1, -{Straight Barb[length=3pt]}},
         LA/.default=1pt
        }

\newcommand{\keywords}[1]{\par\noindent\textbf{Keywords:} #1}
\newcommand{\msc}[1]{\par\noindent\textbf{MSC 2020:} #1}

\usepackage{graphicx}%
\usepackage{multirow}%
\usepackage{amsmath,amssymb}%
\usepackage{amsthm}%
\usepackage{mathrsfs}%
\usepackage[title]{appendix}%
\usepackage{xcolor}%
\usepackage{textcomp}%
\usepackage{manyfoot}%
\usepackage{booktabs}%
\usepackage{algorithm}%
\usepackage{algorithmicx}%
\usepackage{algpseudocode}%
\usepackage{listings}%

\makeatother

\newtheorem{theorem}{Theorem}[section]
\newtheorem{lemma}[theorem]{Lemma}
\newtheorem{definition}[theorem]{Definition}
\newtheorem{example}[theorem]{Example}
\newtheorem{remark}[theorem]{Remark}
\newtheorem{corollary}[theorem]{Corollary}
\newtheorem{proposition}[theorem]{Proposition}

\title{Simplicial Approach to Frobenius Algebras in the Category of Relations}
\author{Dominik Lachman\thanks{Department of Algebra and Geometry,\\
Palacký University Olomouc,\\
17. Listopadu 12,\\
Czech Republic,\\
\texttt{dominik.lachman@upol.cz}}%
\space\thanks{Author acknowledges the support by the Czech Science Foundation (GAČR): project 24-14386L}} 

\begin{document}

\maketitle

\begin{abstract}
    Frobenius algebras in the category of sets and relations
($\Rel$) serve as a unifying framework for various algebraic and combinatorial structures, including groupoids, effect algebras, and abstract circles. Recently, a nerve construction of simplicial sets for Frobenius algebras in $\Rel$ has been introduced. In this work, we investigate the lifting properties of these simplicial sets, linking them to the algebraic properties of Frobenius algebras. We introduce $\epsilon$-simplicial sets -- simplicial sets with marked edges -- that enable the representation of a broader class of structures, such as test spaces from quantum logic. Our main results focus on weakly saturated classes generated by cofibrations, corresponding to specific lifting problems. Furthermore, we provide a characterization of Frobenius algebras in $\Rel$ within the framework of $\epsilon$-simplicial sets. These findings lay the groundwork for the development of a convenient model structure in future research.
\end{abstract}

\msc{06C15, 18G30, 18M05}

\maketitle

\section{Introduction}\label{sec1}
The category of sets and relations, denoted by $\Rel$, is studied in connection with a program to translate foundational concepts from quantum mechanics into the language of category theory. With additional structure, $\Rel$ is 
a dagger compact category (see~\cite{HV}), providing a simplified framework for studying categorical analogues of quantum mechanical principles. Within this setting, Frobenius algebras play a crucial role. A Frobenius algebra in a category is an object equipped with compatible monoid and comonoid structures that satisfy the Frobenius identity. In the case of $\Rel$, the category of Frobenius algebras, denoted $\RelFrob$, contains several important structures, in particular, \emph{groupoids} (see~\cite{CCH}) and \emph{effect algebras} (see~\cite{PS}).

Another motivation for studying Frobenius algebras in
$\Rel$ arises from their connections to \emph{2-dimensional quantum field theories} (see \cite{JK}). Within this line of research, Mehta and Zhang~\cite{MZ} associate a simplicial set with each Frobenius algebra in $\Rel$ (see also~\cite{CKM} and \cite{CMS}). The construction parallels the classical nerve construction for small categories, particularly groupoids, where simplicial sets serve as a combinatorial topological model for categorical structures. Notably, simplicial sets called \emph{quasi-categories} have a central role in modeling higher categories, supported by a rich and well-established homotopy theory~\cite{L}.
This interplay highlights the potential of using simplicial sets to explore the algebraic and combinatorial properties of potentially higher Frobenius algebras internal to categories of interest.

In contrast to groupoids, representing effect algebras as simplicial sets appears to be a novel idea, although not entirely unprecedented. 
The particular case of \emph{orthomodular posets} has been studied in relation to the category of \emph{test spaces} introduced by Foulis and Randall (see, e.g., \cite{FPR}), which are used to study phenomena such as \emph{contextuality}, \emph{locality}, and \emph{non-locality}.
As noted by Foulis in~\cite{F}, a test space is essentially a combinatorial-simplicial complex where some simplices (so-called tests) are marked. 
Moreover, it has been shown~\cite{SU, Je} that the category of effect algebras can be fully and faithfully embedded into a certain category of presheaves over finite Boolean algebras.

In the study of cohomology of effect algebras~\cite{R}, it is observed that effect algebras and the so-called \emph{abstract circles} share certain formal properties, leading to the introduction of \emph{effect algebroids} as a unifying generalization. Notably, all these structures can be organized as Frobenius algebras in $\Rel$ (Example~\ref{ex:EAlg}), suggesting that $\RelFrob$ provides a natural setting for a broad class of algebraic-combinatorial structures.

This article builds on the results of~\cite{MZ}. The primary goal of this article is to translate the essential algebraic properties of a Frobenius algebra $\mathcal{F}$ into the combinatorial framework of the simplicial set $N(\mathcal{F})$. 
For instance, we describe the \emph{lifting properties} of $N(\mathcal{F})$ corresponding to the Frobenius identity, rotation, cancellation property, and so on. Then we investigate the \emph{weakly saturated classes} generated by the lifting problems in question. We establish a foundation that can be utilized in future work to construct a suitable model structure, enabling the representation of (potentially higher) Frobenius algebras as particular fibrant objects.

Although some of our results are parallel to those in~\cite{MZ}, our approach is different in an essential way.
In~\cite{MZ}, Frobenius algebras are characterized through simplicial sets with a $\mathbb{Z}$-action on simplices, referred to as \emph{rotation}, thus forming \emph{paracyclic sets}.  In contrast, our approach considers specific simplicial sets in which certain edges are marked (corresponding to counit elements); we call these \emph{$\epsilon$-simplicial sets}. This framework offers several advantages: working with $\epsilon$-simplicial sets is arguably simpler than working with paracyclic sets. Additionally, the category of $\epsilon$-simplicial sets consists of presheaves over a Reedy category, allowing us to apply related machinery in future research.

In the case of $\epsilon$-simplicial sets, the rotation ($\mathbb{Z}$-action) on simplices is captured via lifting properties, suggesting that it can be defined up to a form of 'homotopy' equivalence. This perspective is particularly advantageous for capturing \emph{test spaces}, where each element - called an event - may have multiple rotations (called local orthocomplements in this context), all of which represent the same element within the corresponding ``homotopy'' effect algebra. Given this inherent ambiguity, it is natural to organize test spaces as $\epsilon$-simplicial sets rather than paracyclic sets.


The article is organized as follows. In Section 2, we introduce the definition of Frobenius algebras in $\Rel$ and state some of their basic properties. Section 3 provides a brief overview of simplicial sets and describes the nerve construction of a monoid in $\Rel$. 
In Section 4, we present several examples of algebraic properties that can be captured using extension properties. Section 5 introduces the concept of $\epsilon$-simplicial sets and $\epsilon$-horns, translating essential properties of Frobenius algebras into simplicial combinatorics. Section 6 is the technical core of the article. We give several theorems that focus on a weakly saturated class generated by the simplicial morphisms in question. We explore the relationships between the Frobenius identity and associativity. In Section 7, we characterize the $\epsilon$-simplicial sets that arise from the nerve functor. Finally, in the concluding section, we discuss the aims of future research.

\section{Frobenius algebras in $\Rel$}
In this article, we consider $\Rel$ as a monoidal category, where objects are sets, arrows are binary relations, and the monoidal structure is given by the cartesian product $\times$ and any one-element set $\{\bullet\}$ taken as the unit. In particular, we can use string diagrams to formulate equations. We assume basic knowledge of string diagrams. For more details, see~\cite{HV}.

In this section, we introduce the concept of Frobenius algebras in $\Rel$ and give some important examples.

\begin{definition}\label{def:monoid}
    Let $X$ be a set and $\mu\colon X\times X\dashrightarrow X$, $\eta\colon\{\bullet\}\dashrightarrow X$ be relations considered as arrows in $\Rel$. We say that $(X;\mu,\eta)$ is a monoid if it satisfies 
    \begin{enumerate}[(i)]
        \item $\mu\circ (\mu\times\mathrm{id})=\mu\circ(\mu\times\mathrm{id})$ (associativity),
        \item $\mu\circ (\eta\times \mathrm{id})=\mathrm{id}=\mu\circ(\mathrm{id}\times \eta)$ (unit).
    \end{enumerate}
    In (ii), we identify $\eta\times\mathrm{id}$ with the relation $X\cong \{\bullet\}\times X\dashrightarrow X\times X$. Similarly, in the case of $\mathrm{id}\times \eta$.
\end{definition}
There are various concepts of homomorphisms between monoids in $\Rel$. In this article, 
We call a mapping $h\colon X_1\rightarrow X_2$ a homomorphism $(X_1,\mu_1,\eta_1)\rightarrow (X_2,\mu_2,\eta_2)$ if, for each $a,b,c\in X_1$, we have
\begin{enumerate}[(i)]
    \item $a\in \eta_1 \implies h(a)\in\eta_2$,
    \item $\mu_1\colon(a,b)\dashmapsto c \implies \mu_2\colon(h(a),h(b))\dashmapsto h(c)$.
\end{enumerate}
The resulting category is denoted by $\RelMon$.

Analogously, we define a comonoid in $\Rel$ as a monoid in $\Rel^{op}$.
\begin{definition}
    Let $X$ be a set, and let $\delta\colon X\dashrightarrow X\times X$, $\epsilon\colon X\dashrightarrow\{\bullet\}$ be relations. We say a triple $(X;\delta,\epsilon)$ is a comonoid in $\Rel$ if
\begin{enumerate}[(i)]
        \item $(\delta\times\mathrm{id})\circ \delta= (\mathrm{id}\times \delta)\circ\delta$ (co-associativity),
        \item $(\epsilon\times \mathrm{id})\circ \delta=\mathrm{id}=(\mathrm{id}\times \epsilon)\circ\delta$ (co-unit).
    \end{enumerate}
\end{definition}
We will picture the algebraic data $\mu$, $\eta$, $\delta$, and $\epsilon$ using string diagrams as in Figure~\ref{fig:01}.
\begin{figure}[H]
    \centering
    \includegraphics[scale=0.6]{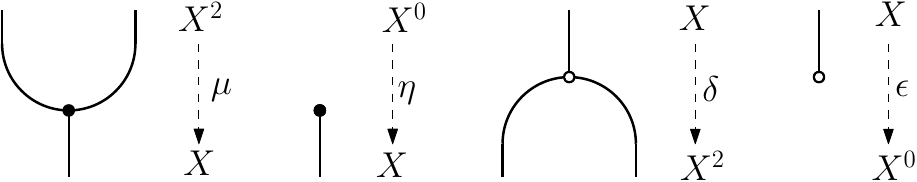}
    \caption{String diagrams for a Frobenius structure}
    \label{fig:01}
\end{figure}
In a monoid, every element has a unique left and right unit.
\begin{lemma}[\cite{MZ}, Lemma 3.15]\label{lem:sourceAndTarget}
    Let $(X;\mu,\eta)$ be a monoid in $\Rel$. Then for each $a\in X$, there are unique unit elements $s(a),t(a)\in\eta$ such that $\mu\colon (t(a),a)\dashmapsto a$ and $\mu\colon(a,s(a))\dashmapsto a$.
\end{lemma}
Given an element $a$ of a monoid in $\Rel$, we will call $t(a)$ the \emph{target} of $a$ and $s(a)$ the \emph{source} of $a$.
\begin{definition}
A Frobenius algebra in $\Rel$ is an algebra $\mathcal{F}=(X;\mu,\eta,\delta,\epsilon)$ where $(X;\mu,\eta)$ is a monoid in $\Rel$, $(X;\delta,\epsilon)$ is a comonoid in $\Rel$, and the Frobenius identity holds
\begin{align}\label{eq:frobId}
    (\mu,\mathrm{id})\circ(\mathrm{id},\delta)=\delta\circ\mu=(\mathrm{id},\mu)\circ(\delta,\mathrm{id}).
\end{align}
\end{definition}
\begin{figure}[H]
    \centering
  \includegraphics[scale=0.6]{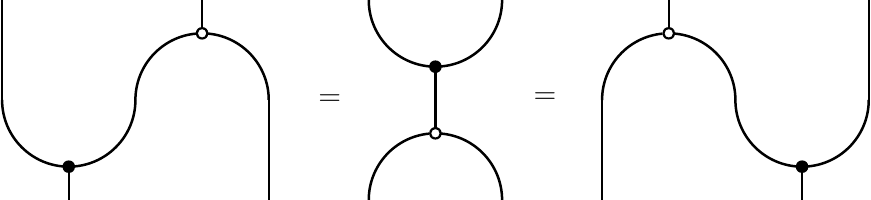}
  \caption{Frobenius identity}
\end{figure}
\begin{remark}\label{rem:Asoc}
The rules of Frobenius identity, unit, and counit together imply associativity and co-associativity (see \cite{MZ}, Lemma~3.12). Hence, we may leave the (co)associativity rule out from the definition of the Frobenius algebra.
\end{remark}
We call a mapping $h\colon\mathcal{F}_1\rightarrow \mathcal{F}_2$ between two Frobenius algebras a \emph{Frobenius algebra homomorphism} if it is a monoid and a comonoid homomorphism. We denote the resulting category by $\RelFrob$.
    \begin{example}[\cite{CCH}]
    A groupoid is a small nonempty category where every morphism is an isomorphism. Given a groupoid $\mathcal{G}$, we can organize $\mathcal{G}$ as a Frobenius algebra. Let $X$ be a set of all morphisms in $\mathcal{G}$, and for each $f,g,h\in X$, we set
    \begin{itemize}
        \item $\mu\colon (f,g)\dashmapsto h$ iff $f\circ g= h$ iff $\delta\colon h \dashmapsto (f,g)$,
        \item $f\in \eta$ iff $\mathcal{F}$ is an identity iff $f\in\epsilon$.
    \end{itemize}
\end{example}
\begin{example}[\cite{PS}]
    An effect algebra is a partial algebra $E=(E;\oplus,',0,1)$, where 
    $(E;\oplus,0)$ is a partial commutative monoid, $'$ is a unary operation, and $1$ is a constant so that:
    \begin{enumerate}[(i)]
        \item For each $a\in E$, $a'$ is a unique element such that $a\oplus a'=1$.
        \item For each $a\in E$, if $a\oplus 1$ is defined, then $a=0$.
    \end{enumerate}
We can organize an effect algebra $E$ as a Frobenius algebra as follows. For each $a,b,c\in E$, we set
\begin{itemize}
    \item $\mu\colon (a,b)\dashmapsto c$ iff $a\oplus b=c$,
    \item $a\in\eta$ iff $a=0$,
    \item $\delta\colon a\dashmapsto (b,c)$ iff $a=(b'\oplus c')'$,
    \item $a\in\epsilon$ iff $a=1$.
\end{itemize}
\end{example}
The following structure is introduced in~\cite{R} as a common generalization of effect algebras and so-called abstract circles:
\begin{example}\label{ex:EAlg}
    An effect algebroid consists of a non-empty set $P$ (of points) and for each $x,y,z\in P$ the following data: a set $\mathrm{Hom}(x,y)$ (of segments), mappings 
    $$\cup\colon\mathrm{Hom}(x,y)\times\mathrm{Hom}(y,z)\rightarrow \mathrm{Hom}(x,z), \quad(-)^{\perp}\colon \mathrm{Hom}(x,y)\rightarrow\mathrm{Hom}(y,x),$$ and $0_x,1_x\in\mathrm{Hom}(x,x)$, such that:
    \begin{itemize}
        \item All segments $\bigcup_{x,y}\mathrm{Hom}(x,y)$ together with $\cup$ and $0_x$'s naturally form a partial monoid (in $\Rel$).
        \item For each $a\in\mathrm{Hom}(x,y)$ and $b\in\mathrm{Hom}(y,x)$, we have
        $$a\cup b=1_x\Leftrightarrow a=b^\perp\Leftrightarrow b=a^\perp.$$
        \item For each $a\in\mathrm{Hom}(x,y)$, if $1_x\cup a$ or $a\cup 1_y$ is defined, then $x=y$ and $a=0_x$.
    \end{itemize}
    It is straightforward to prove that the following data lead to a Frobenius algebra (in $\Rel$) on the set of all segments: $\mu=\cup$, $\eta=\{0_x\mid x\in P\}$, $\epsilon=\{1_x\mid x\in P\}$, and $\delta\colon a\dashmapsto (b,c)$ if and only if $\mu\colon (c^\perp,b^\perp)\dashmapsto a^\perp$.
\end{example}
Given a Frobenius algebra, by $\alpha$ we denote a composition $\alpha:=\epsilon\circ \mu\colon X\times X\dashrightarrow \{\bullet\}$. Clearly, we can identify $\alpha$ with a subset of $X\times X$. In the case of a groupoid $G$, $\alpha$ contains all pairs $(f,f^{-1})$, $f\in G$. In the case of effect algebras $E$, $\alpha$ contains all pairs $(a,a')$, $a\in E$. For a general Frobenius algebra, $\alpha$ is always a graph of a bijection. Moreover, $\mu$, $\eta$, and $\alpha$ together determine $\delta$ and $\epsilon$, as is apparent from the following lemma.
\begin{lemma}\label{lem:alpha}
    Let $\mathcal{F}=(X;\mu,\eta,\delta,\epsilon)$ be a Frobenius algebra in $\Rel$. Then:
    \begin{enumerate}[(i)]
        \item For each $x\in X$, there are unique $\hat{\alpha}(x)$ and $\hat{\beta}(x)$, such that $(x,\hat{\alpha}(x))\in \alpha=\mu\circ\epsilon\subseteq X\times X$ and $(\hat{\beta}(x),x)\in \alpha=\mu\circ\epsilon\subseteq X\times X$. Moreover, the so-defined mappings $\hat{\alpha},\hat{\beta}\colon X\rightarrow X$ are bijections and $\hat{\alpha}^{-1}=\hat{\beta}$.
        \item For each $r\in\eta$, the element $\hat{\alpha}(r)$ is the unique counit element $e\in\epsilon$ such that $t(e)=r$. Similarly, the element $\hat{\beta}(r)$ is the unique counit element $e\in\epsilon$ such that $s(e)=r$. 
        \item For each triple $(x,y,z)\in X^3$, we have 
        \begin{equation}\label{eq:delta}
            \delta\colon x\dashmapsto (y,z)\Leftrightarrow \mu\colon (x,\hat{\alpha}(z))\dashmapsto y \Leftrightarrow \mu\colon (\hat{\beta}(y),x)\dashmapsto z.
        \end{equation}
    \end{enumerate}
\end{lemma}
\begin{proof}
All claims are proved in~\cite{MZ}. In detail, (i) Proposition 3.4, (ii) Proposition 3.5, and (iii) Lemma 3.7. We give here only a proof of (iii) using string diagrams:
\begin{equation*}
\centering
       \includegraphics[scale=0.6]{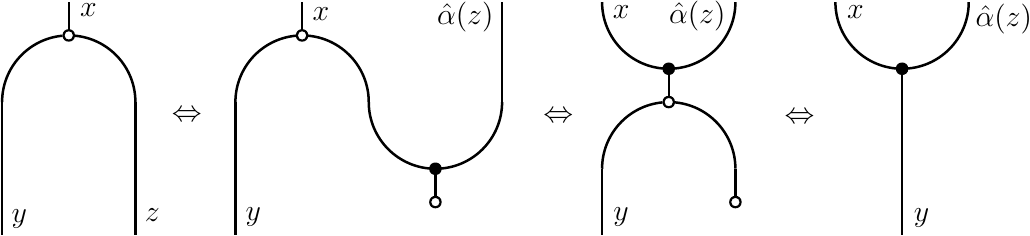} 
\end{equation*}
The first step is due to (ii). The second step follows from the Frobenius identity. The last step is by the counit rule.
\end{proof}
\begin{corollary}\label{cor:frobHom}
    A mapping $h\colon\mathcal{F}_1\rightarrow\mathcal{F}_2$ is a Frobenius algebra homomorphism if and only if it is a monoid homomorphism and preserves counit elements.
\end{corollary}
\begin{proof}
    If $h$ preserves $\mu$ and $\epsilon$, it also preserves $\hat{\alpha}$. By~\eqref{eq:delta}, $h$ preserves $\delta$ as well.
\end{proof}
\section{Simplicial sets}\label{sec:3}
In this section, we briefly recall the basic notation of simplicial sets, and describe a (nerve) functor from the category of monoids in $\Rel$ to the category of simplicial sets. Let $\Delta$ be a category where objects are finite chains $[n]=\{0<1<\cdots<n\}$, $n\in\mathbb{N}_0$, and morphisms are non-decreasing mappings. A category of \emph{simplicial sets} is the category of presheaves $\SSet=\mathrm{Set}^{\Delta^{op}}$. The standard $n$-simplex $\Delta^n$, $n\in\mathbb{N}$, is the image of $[n]$ under the Yoneda embedding
\begin{equation*}
    \Delta^n=\mathrm{Hom}_\Delta(-,[n]).
\end{equation*}
Given a morphism $f\colon [n]\rightarrow[m]$ in $\Delta$, we will encode $f$ as an $(n+1)$-tuple $$f=\langle f(0),f(1),\ldots, f(n)\rangle.$$
The $k$-dimensional faces of the standard $n$-simplex $\Delta^n$ correspond (via Yoneda embedding) to the $\binom{n}{k}$ embeddings 
\begin{align}\label{eq:embed}
\langle i_0,i_1,\ldots,i_k\rangle\colon [k]\hookrightarrow [n],\text{\ for\ each\ } 0\leq i_0<\cdots<i_k\leq n.    
\end{align}
See the following example of $\Delta^2$:

\begin{align}\label{eq:exOfSim}
              \begin{tikzcd}[row sep={3em,between origins}, column sep={4.1em,between origins}, ampersand replacement=\&, every label/.append style = {font = \small},font=\small]
    \&
    {\langle 2\rangle}
    \&\\[10pt]\&{\langle 012\rangle}
    \&\\
    {\langle 0\rangle}\arrow[ruu,"\langle 02\rangle"{description}]\arrow[rr,"\langle 01\rangle"{description}]
    \&\&{\langle 1\rangle}\arrow[luu,"\langle 12\rangle"{description}]
\end{tikzcd}
\end{align}

The subsimplicial set of $\Delta^n$ induced (by Yoneda embedding) 
from~\eqref{eq:embed} we denote by $\Delta^{\{i_0,\ldots,i_k\}}\subseteq \Delta^n$. In the particular
case of $\langle 0,1,\ldots,\hat{i},\ldots,n\rangle$, we denote the
corresponding face by $\partial_i\Delta^n$. Geometrically, we think of
$\partial_i\Delta^n$ as the unique codimension $1$ face of $\Delta^n$ which does not contain the $i$-th vertex.

Given a general simplicial set $X$, we refer to the set of its $n$-simplices as $X_n:=X([n])$. Moreover, $0$-simplices are called vertices and $1$-simplices are called edges. 
Let $0\leq i\leq n$, we define
\begin{align*}
    \delta^n_i&:=\langle0,\ldots,i-1,i+1,\ldots,n\rangle\colon [n-1]\rightarrow [n],\\
    \sigma^n_i&:=\langle0,\ldots,i,i,\ldots,n\rangle\colon[n+1]\rightarrow [n].
\end{align*}
We call the mappings 
$$s_i^n=X(\sigma^n_i)\colon X_{n-1}\rightarrow X_n \quad\text{and}\quad d_i^n=X(\delta^n_i)\colon X_{n}\rightarrow X_{n-1}$$
\emph{face map} and \emph{degeneracy map}, respectively.
We will the standard notation for a \emph{boundary} and \emph{horns}:
\begin{align}\label{eq:faceAndHorn}
    \partial \Delta^n&=\bigcup_{i=0}^n\partial_i\Delta^n,\\
    \Lambda^n_j&=\bigcup_{i\not= j}\partial_i\Delta^n, \text{\ for\ }j=0,\ldots,n.
\end{align}

\begin{example}\label{ex:nerv}
    Let $\mathcal{C}$ be a small category. Since each poset is a category, we may think of $\Delta$ as a category of small categories. Hence, there is a simplicial set $N(C)=\Hom(-,C)$. The set of $n$-simplices $N(C)_n=\Hom([n],C)$ consists of $n$-tuples of composable arrows in $\mathcal{C}$.
\end{example}
Since each small category is also a monoid in $\Rel$, there is a (full) embedding $\Delta\hookrightarrow \RelMon$. Hence, following the same pattern as in Example~\ref{ex:nerv}, we can define for each $\mathcal{M}$ in $\RelMon$ a simplicial set
\begin{align}\label{eq:nerv}
    N(\mathcal{M}):=\Hom(-,\mathcal{M}).
\end{align}
Note that the prescription~\eqref{eq:nerv} is functorial in $\mathcal{M}$.
Given $\mathcal{M}=(M;\mu,\eta)$ a monoid in $\Rel$, we can define low-dimensional simplices of $X:=N(\mathcal{M})$ in more explicit terms as follows:
\begin{itemize}
    \item $X_0=\eta$,
    \item $X_1=M$,
    \item $X_2=\{(a,b,c)\in X^3\mid \mu\colon (a,b)\dashmapsto c\}$.
\end{itemize}
The face and degeneracy maps then act as:
\begin{itemize}
    \item $s^0_0\colon X_0\rightarrow X_1$ equals $\eta\subseteq M$,
    \item $d^1_0\colon a\mapsto t(a)$ and $d^1_1\colon x\mapsto s(a)$,
    \item $s^1_0\colon x\mapsto (a,s(a),a)$ and $s^1_1\colon a\mapsto (t(a),a,a)$,
    \item $d^2_0\colon (a,b,c)\mapsto a$, $d^2_1\colon (a,b,c)\mapsto c$, and $d^2_2\colon (a,b,c)\mapsto b$.
\end{itemize}
The above explicit construction is used to define $N(\mathcal{M})$ in~\cite{MZ} (for the case of a Frobenius algebra).

\begin{theorem}\label{thm:2coskeletal}
    The construction above results in a $2$-coskeletal simplicial set $X$. That is, for each $n\geq 3$ and $h\colon\partial \Delta^n\hookrightarrow X$, there is a unique extension $\hat{h}\colon \Delta^n\rightarrow X$. Moreover, if some $g\colon\partial\Delta^2\rightarrow X$ has an extension $\hat{g}\colon\Delta^2\rightarrow X$, then the extension is unique.

    Consequently, given $\mathcal{M}, \mathcal{N}\in\RelMon$, each natural transformation $\psi\colon N(\mathcal{M})\rightarrow N(\mathcal{N})$ is uniquely determined by its action on the edges.
\end{theorem}
\begin{proof}
    Assume any $h\colon\partial\Delta^n\rightarrow N(\mathcal{M})$ with $n\geq 3$. Since $\partial\Delta^n$ contains all the $i$ -simplices of $\Delta^n$, for $i=0,1,2$, it follows that $h$ defines a unique monoid homomorphism $\overline{h}\colon [n]\rightarrow \mathcal{M}$. This $\overline{h}$ is the unique extension $\Delta^n\rightarrow N(\mathcal{M})$.

    The moreover part is clear from the concrete description of $X_2$. The last part of the statement follows from the two previous claims.
\end{proof}
Note that the following proposition is equivalent to the claim that the subcategory $\Delta\subset \RelMon$ is a dense subcategory.
\begin{proposition}\label{prop:fullfaithful}
    The functor $N\colon\mathcal{M}\mapsto\mathrm{Hom}(-,\mathcal{M})$ from $\RelMon$ to the category of simplicial sets is full and faithful.
\end{proposition}
\begin{proof}
    Let $\mathcal{M},\mathcal{N}\in\RelMon$. In the proof, we will identify an element $a\in\mathcal{M}$ with the corresponding unique morphism $a\colon[1]\rightarrow\mathcal{M}$.

    \emph{Faithfulness:} Assume $f,g\colon\mathcal{M}\rightarrow\mathcal{N}$. If $f(a)\not=g(a)$ for some $a\in\mathcal{M}$, then $f\circ a\not= g\circ a$. Hence $N(f)\not=N(g)$.

\emph{Fullness:} Assume a natural transformation $$\psi\colon \Hom(-,\mathcal{M})\rightarrow\Hom(-,\mathcal{N}).$$ We need to find $f\colon\mathcal{M}\rightarrow\mathcal{N}$ so that $\psi=N(f)$. For each $a\colon[1]\rightarrow\mathcal{M}$, we set $f(a)=\psi_1(a)$. We will show that $f$ is a homomorphism. Assume $a_0,a_1,a_2\in\mathcal{M}$ satisfy 
\begin{equation}\label{eq:tri}
    \mu_\mathcal{M}\colon(a_0,a_2)\dashmapsto a_1
\end{equation}
and let $t\colon [2]\rightarrow\mathcal{M}$ be the obvious morphism coding~\eqref{eq:tri}. That is, $a_i=d_i^2(t)$, for $i=0,1,2$. It follows $\psi_2(t)\colon [2]\rightarrow\mathcal{N}$ witnesses $\mu_\mathcal{N}\colon(f(a_0),f(a_2))\dashmapsto f(a_1)$. 

To prove that $f$ preserves unit elements, observe that $r\in\mathcal{M}$ is a unit if and only if $r\colon[1]\rightarrow\mathcal{M}$ splits through a monoid of one element as $[1]\rightarrow[0]\rightarrow\mathcal{M}$. That is, if and only if $r$ is a degenerate edge (in the image of $s_0^0$). Since $\psi$ preserves degenerate edges, the mapping $f$ preserves unit elements.

Finally, since $\psi$ coincides with $N(f)$ on the edges (by the very definition of $f$), we have $\psi=N(f)$ by Theorem~\ref{thm:2coskeletal}.
\end{proof}

When working with a simplicial set $X$, we will frequently refer to its $n$-simplices $x\colon\Delta^n\rightarrow X$ using diagrams as follows (here $n=3$):
\begin{align}
      \begin{tikzcd}[row sep={2.4em,between origins}, column sep={3.3em,between origins}, ampersand replacement=\&]
    \&
    {3}
    \&\\[10pt]\&
    {2}\arrow[u,"c"{description}]
    \&\\
    {0}
    \arrow[ruu,"e"{description}, bend left=5]
    \arrow[rr, "a"{description}, bend right=5]
    \arrow[ru, "{x}"{description}]
    \&\&
    {1}\arrow[uul,"y"{description}, bend right=5]\arrow[lu, "b"'{description,name=f13}]
\end{tikzcd} 
\end{align}
Usually, we will label only edges somehow important to the situation. Since vertices are uniquely determined by edges, we will label them only by numbers to stress the order.
\section{Lifting properties and algebraic properties}
In this section, we recall the notion of lifting properties and present several algebraic properties of monoids that can be captured using lifting properties.

Let $m\colon A\rightarrow B$ and $p\colon X\rightarrow Y$ be two arrows in a category $\mathcal{C}$. Assume that for each commutative square 
\begin{align}
    \begin{tikzcd}[ampersand replacement=\&]
        {A}\arrow[d,"m"]\arrow[r,"f"]\&{X}\arrow[d,"p"]\\
        {B}\arrow[r,"g"]\&{Y}
    \end{tikzcd}
\end{align}
there is an arrow $l\colon B\rightarrow X$, such that $f=l\circ m$ and $g=p\circ l$. In this case, we say $m$ has the \emph{left lifting property} (LLP) for $p$ or $p$ has the \emph{right lifting property} (RLP) for $m$. Denote this binary relation by 
\begin{align}\label{eq:lift}
    m\boxslash p.
\end{align}
In the particular case where $Y$ is a terminal object of the category $\mathcal{C}$ and $p=!_X$ is the unique arrow, we have $m\boxslash p$ if and only if 
\begin{align*}
    \Hom(m,X)\colon\Hom(B,X)\rightarrow\Hom(A,X)
\end{align*}
is a surjection. 
In such a case, we abbreviate~\eqref{eq:lift} as $m\boxslash X$, and we say that $X$ has the \emph{extension property} with respect to $m$.
If $m$ is embedding and $B$ is clear from the context, we will also say $X$ \emph{admits a filling of} $A$. 
\begin{example}\label{ex:assoc}
    Let $(M;\mu,\eta)$ be a monoid in $\Rel$ and $X=N(\mathcal{M})$ the corresponding simplicial set. Then the associativity rule guarantees that $X$ has the extension property with respect to $\partial_3\Delta^3\cup\partial_1\Delta^3\hookrightarrow \Delta^3$ and $\partial_2\Delta^3\cup\partial_0\Delta^3\hookrightarrow \Delta^3$.
 \begin{align*}
   \begin{tikzcd}[row sep={2.4em,between origins}, column sep={3.3em,between origins}, ampersand replacement=\&,execute at end picture={
\foreach \Nombre in  {A,B,C}
  {\coordinate (\Nombre) at (\Nombre.center);}
\fill[pattern=horizontal lines,pattern color=gray!50] 
  (A) -- (B) -- (C) -- cycle;
  \foreach \Nombre in  {A,B,D}
  {\coordinate (\Nombre) at (\Nombre.center);}
\fill[pattern=north east lines,pattern color=gray!50] 
  (A) -- (C) -- (D) -- cycle;
}]
    \&
    |[alias=D]|3 \& \\[10pt]\&
    |[alias=C]|2\arrow[u,"a"{description}]\&\\
    |[alias=A]|0\arrow[ru]
    \arrow[ruu]
    \arrow[rr,"c"{description}]
    \&\&
    |[alias=B]|1\arrow[ul, "b"{description}]\\
    \&\partial_1\Delta^3\cup\partial_3\Delta^3\&
\end{tikzcd} 
\hookrightarrow
  \begin{tikzcd}[row sep={2.9em,between origins}, column sep={3.3em,between origins}, ampersand replacement=\&,
execute at end picture={
  \foreach \Nombre in {A,B,D}
    {\coordinate (\Nombre) at (\Nombre.center);}
    \fill[opacity=0.1] (A) -- (B) -- (D) -- cycle;
}]
  \&
  |[alias=D]|3 \& \\
  \&
  |[alias=C]|2 \arrow[u,"a"{description}] \& \\
  |[alias=A]|0 \arrow[ru]
               \arrow[ruu]
               \arrow[rr,"c"{description}]
  \& \&
  |[alias=B]|1 \arrow[luu] \arrow[ul,"b"{description}] \\
  \&
  \Delta^3 \&
\end{tikzcd}
\hookleftarrow
\begin{tikzcd}[row sep={2.4em,between origins}, column sep={3.3em,between origins}, ampersand replacement=\&,execute at end picture={
\foreach \Nombre in  {B,C,D}
  {\coordinate (\Nombre) at (\Nombre.center);}
\fill[pattern= north west lines,pattern color=gray!50] 
  (B) -- (C) -- (D) -- cycle;
  \foreach \Nombre in  {A,B,D}
  {\coordinate (\Nombre) at (\Nombre.center);}
\fill[pattern=vertical lines,pattern color=gray!50] 
  (A) -- (B) -- (D) -- cycle;
}]
    \&
    |[alias=D]|3 \& \\[10pt]\&
    |[alias=C]|2\arrow[u,"a"{description}]\&\\
    |[alias=A]|0
    \arrow[ruu]
    \arrow[rr,"c"{description}]
    \&\&
    |[alias=B]|1\arrow[luu]\arrow[ul, "b"{description}]\\
    \&\partial_0\Delta^3\cup\partial_2\Delta^3\&
\end{tikzcd} 
\end{align*}
\end{example}

The following theorem provides an example of a correspondence between certain properties of monoids and lifting properties.
\begin{theorem}\label{thm:alfAndLif}
    Let $\mathcal{M}=(M;\mu,\eta)$ be a monoid in $\Rel$ and $N(\mathcal{M})$ be the corresponding simplicial set. Then 
    \begin{enumerate}[(i)]
        \item $\mu$ has cancellation property in the first coordinate if and only if 
        \begin{equation}\label{eq:lif1}
            (\Lambda_0^3\hookrightarrow \Delta^3)\boxslash N(\mathcal{M}).
        \end{equation}
        \item $\mu$ has cancellation property in the second coordinate if and only if 
        \begin{equation}\label{eq:lif2}
           (\Lambda_3^3\hookrightarrow \Delta^3)\boxslash N(\mathcal{M}).
        \end{equation}
        \item $\mu$ is a partial operation if and only if
        \begin{equation}\label{eq:lif3}
            (\Lambda_1^3\hookrightarrow \Delta^3)\boxslash N(\mathcal{M})\text{\ \ \ or\ equivalenty\ iff\ \ \ }
        (\Lambda_2^3\hookrightarrow \Delta^3)\boxslash N(\mathcal{M}).
        \end{equation}
    \end{enumerate}
\end{theorem}
\begin{proof}
    Denote $X=N(\mathcal{M})$. We will prove (i) and (ii) together. Assume first that the lifting property~\eqref{eq:lif1} (\eqref{eq:lif2}, respectively) holds and for some $a,a_1,a_2,b,b_1,b_2,c\in\mathcal{M}$ we have $\mu\colon (a_1,b)\dashmapsto c$ and $\mu\colon (a_2,b)\dashmapsto c$ ($\mu\colon (a,b_1)\dashmapsto c$ and $\mu\colon (a,b_2)\dashmapsto c$, respectively). Then we have a horn $\Lambda_0^3\hookrightarrow X$ ($\Lambda_3^3\hookrightarrow X$, respectively) given by~\eqref{diag:canc}. (Note that in~\eqref{diag:canc}, the missing faces are highlighted by a hatching.)
\begin{align}\label{diag:canc}
       \begin{tikzcd}[row sep={2.4em,between origins}, column sep={3.3em,between origins}, ampersand replacement=\&,execute at end picture={
  \foreach \Nombre in  {B,C,D}
  {\coordinate (\Nombre) at (\Nombre.center);}
\fill[pattern=north east lines,pattern color=gray!50] 
  (B) -- (C) -- (D) -- cycle;
}]
    \&
    |[alias=D]|3 \& \\[10pt]\&
    |[alias=C]|2\arrow[u,"t(c)"{description}]\&\\
    |[alias=A]|0\arrow[ru,"c"{description}]
    \arrow[ruu,bend left= 5,"c"{description}]
    \arrow[rr,bend right =5,"b"{description}]
    \&\&
    |[alias=B]|1\arrow[luu,bend right = 5,"a_2"{description}]\arrow[ul, "a_1"{description}]
\end{tikzcd} \ \ \ \text{respective}\ \ \ 
       \begin{tikzcd}[row sep={2.4em,between origins}, column sep={3.3em,between origins}, ampersand replacement=\&,execute at end picture={
  \foreach \Nombre in  {A,B,C}
  {\coordinate (\Nombre) at (\Nombre.center);}
\fill[pattern=vertical lines,pattern color=gray!50] 
  (B) -- (C) -- (A) -- cycle;
}]
    \&
    3 \& \\[10pt]\&
    |[alias=C]|2\arrow[u,"a"{description}]\&\\
    |[alias=A]|0\arrow[ru,"b_2"{description}]
    \arrow[ruu,bend left= 5,"c"{description}]
    \arrow[rr,bend right =5,"b_1"{description}]
    \&\&
    |[alias=B]|1\arrow[luu,bend right = 5,"a"{description}]\arrow[ul, "s(a)"{description}]
\end{tikzcd}
\end{align}
The corresponding lifts guarantee $\mu\colon (t(c),a_1)\dashmapsto a_2$ ($\mu\colon (s(a),b_1)\dashmapsto b_2$, resp.). As $t(c),s(a)\in\eta$, we get $a_1=a_2$ ($b_1=b_2$, respectively).

Conversely, assume we have lifting problems $\Lambda_0^3\hookrightarrow X$  and $\Lambda_3^3\hookrightarrow X$, with the following labeling of edges:
\begin{align}\label{diag:sim3}
       \begin{tikzcd}[row sep={2.4em,between origins}, column sep={3.3em,between origins}, ampersand replacement=\&,execute at end picture={
  \foreach \Nombre in  {B,C,D}
  {\coordinate (\Nombre) at (\Nombre.center);}
\fill[pattern=north east lines,pattern color=gray!50] 
  (B) -- (C) -- (D) -- cycle;
}]
    \&
    |[alias=D]|3 \& \\[10pt]\&
    |[alias=C]|2\arrow[u,"c"{description}]\&\\
    |[alias=A]|0\arrow[ru,"e"{description}]
    \arrow[ruu,bend left= 5,"d"{description}]
    \arrow[rr,bend right =5,"a"{description}]
    \&\&
    |[alias=B]|1\arrow[luu,bend right = 5,"f"{description}]\arrow[ul, "b"{description}]
\end{tikzcd} \ \ \ 
       \begin{tikzcd}[row sep={2.4em,between origins}, column sep={3.3em,between origins}, ampersand replacement=\&,execute at end picture={
  \foreach \Nombre in  {A,B,C}
  {\coordinate (\Nombre) at (\Nombre.center);}
\fill[pattern=vertical lines,pattern color=gray!50] 
  (B) -- (C) -- (A) -- cycle;
}]
    \&
    3 \& \\[10pt]\&
    |[alias=C]|2\arrow[u,"c"{description}]\&\\
    |[alias=A]|0\arrow[ru,"e"{description}]
    \arrow[ruu,bend left= 5,"d"{description}]
    \arrow[rr,bend right =5,"a"{description}]
    \&\&
    |[alias=B]|1\arrow[luu,bend right = 5,"f"{description}]\arrow[ul, "b"{description}]
\end{tikzcd}
\end{align}
In the first case, the presence of the three faces gives (1) $\mu\colon (b, a)\dashmapsto e$, (2) $\mu\colon(c,e)\dashmapsto d$, and (3) $\mu\colon (f, a)\dashmapsto d$. Applying associativity to (1) and (2), we yield $g\in M$ such that (4) $\mu\colon (g, a)\dashmapsto d$ and (5) $\mu\colon (c,b)\dashmapsto g$. Due to the property of left cancellation, (4) and (3) give $f=g$, therefore (5) gives the missing face and by 2-coskeletality (Theorem~\ref{thm:2coskeletal}) the 3-simplex in concern admits a filling. The second case of $\Lambda_3^3\hookrightarrow X$ is analogous.

(iii) Assume that $\mu\colon (a,b)\dashmapsto c_1$, $\mu\colon (a,b)\dashmapsto c_2$, and at least one of the lifting properties in~\eqref{eq:lif3} holds. We have two options to prove $c_1=c_2$. We may arrange a lifting problem $\Lambda_1^3\hookrightarrow X$ (left) or a lifting problem $\Lambda_2^3\hookrightarrow X$ (right) as in~\eqref{diag:part}:
\begin{align}\label{diag:part}
       \begin{tikzcd}[row sep={2.4em,between origins}, column sep={3.3em,between origins}, ampersand replacement=\&,execute at end picture={
  \foreach \Nombre in  {A,C,D}
  {\coordinate (\Nombre) at (\Nombre.center);}
\fill[pattern=vertical lines,pattern color=gray!50] 
  (A) -- (C) -- (D) -- cycle;
}]
    \&
    |[alias=D]|3 \& \\[10pt]\&
    |[alias=C]|2\arrow[u,"t(a)"{description}]\&\\
    |[alias=A]|0\arrow[ru,"c_2"{description}]
    \arrow[ruu,bend left= 5,"c_1"{description}]
    \arrow[rr,bend right =5,"b"{description}]
    \&\&
    1\arrow[luu,bend right = 5,"a"{description}]\arrow[ul, "a"{description}]
\end{tikzcd} \ \ \ 
       \begin{tikzcd}[row sep={2.4em,between origins}, column sep={3.3em,between origins}, ampersand replacement=\&,execute at end picture={
  \foreach \Nombre in  {A,B,C}
  {\coordinate (\Nombre) at (\Nombre.center);}
\fill[pattern=vertical lines,pattern color=gray!50] 
  (B) -- (C) -- (A) -- cycle;
}]
    \&
    |[alias=C]|3 \& \\[10pt]\&
    2\arrow[u,"a"{description}]\&\\
    |[alias=A]|0\arrow[ru,"b"{description}]
    \arrow[ruu,bend left= 5,"c_1"{description}]
    \arrow[rr,bend right =5,"s(b)"{description}]
    \&\&
    |[alias=B]|1\arrow[luu,bend right = 5,"c_2"{description}]\arrow[ul, "b"{description}]
\end{tikzcd}
\end{align}
A solution of any of the two lifting problems gives us $c_1=c_2$.

Conversely, assume $\mu$ is a partial operation and consider a lifting problem $\Lambda_1^3\rightarrow X$ with edges labeled as in~\eqref{diag:sim3} (however, this time we miss the face opposite to $1$). Let $A\in X_3$ be a $3$-simplex which we obtain by applying associativity (see Example~\ref{ex:assoc}) to $\partial_0\Delta^3\cup\partial_2\Delta^3\subset \Lambda_1^3$. The face $d^3_3(A)\in X_2$ of $A$ yields an edge $g$, such that $\mu\colon(b,a)\dashmapsto g$. As $\mu$ is a partial operation, $g$ equals $e$, so $A$ is the desired solution by Theorem~\ref{thm:2coskeletal} . The second case of the lifting problem $\Lambda_2^3\rightarrow X$ is analogous.
\end{proof}

\section{$\epsilon$-simplicial sets}\label{sec:eHorn}
In the previous section, we observed that we can organize data of a relational monoid as a simplicial set. For Frobenius algebras, we need more structure. According to Lemma~\ref{lem:alpha}, a Frobenius algebra is determined by its monoidal structure and its counit $\epsilon$. Hence, a convenient way is to consider simplicial sets where some edges (corresponding to counit elements) are marked; we will call them \emph{$\epsilon$-simplicial sets}. 

In this section, we associate an $\epsilon$-simplicial set with each Frobenius algebra in $\Rel$ and describe the lifting properties corresponding to essential properties of Frobenius algebras.
\begin{definition}
    By $\Sim$ we denote a category that arises from $\Delta$ by adding one additional object $[\eps]$, an extra arrow $w\colon [1]\rightarrow [\eps]$, and all the induced compositions. A presheaf $X\colon\Sim^{\text{op}}\rightarrow \mathrm{Set}$ we call an $\epsilon$-simplicial set, and the category of all $\epsilon$-simplicial sets we denote by $\eSSet$.
\end{definition}
Observe that $\Delta$ is a full subcategory of $\Sim$. For each $n\geq 0$, we have $\Hom([n],[\eps])=\{w\circ h \mid h\colon [n]\rightarrow [1]\}$ and $\Hom([\eps],[n])=\emptyset$. Hence, an $\epsilon$-simplicial set is essentially a simplicial set $X$ together with a mapping $X_\epsilon\rightarrow X_1$.

Another way of describing $\Sim$ is the following pushout in the category of small categories: $$\Sim=\Delta\coprod_{[1]}\{[1]\rightarrow[\epsilon]\}.$$

Given an $\epsilon$-simplicial set $X$, we call an edge $a$ which equals $X(w)(a')$, for some $a'\in X_\epsilon$, an \emph{$\epsilon$-edge}. We may think of elements of $X_\epsilon$ as witnesses of $\epsilon$-edges. 
Note that for an $\epsilon$-edge $a$, there may be several witnesses in $X_\epsilon$. 
In the diagrams, we will indicate the $\epsilon$-edges using thick arrows.
\begin{example}
    A test space is a pair $\mathcal{T}=(X,\mathcal{T})$, where $X$ is a set, and $\mathcal{T}$ 
    a collection of subsets of $X$, the elements of which are called tests, such that $\bigcup\mathcal{T}=X$.
The subsets of tests are called events: $\mathcal{E}(\mathcal{T})={A\subseteq T\mid T\in\mathcal{T}}$.    Given two disjoint events $A$ and $B$, we say $B$ is a local orthocomplement to $A$ if $A\cup B$ is a test. Further, we define a binary relation $\sim$ on events as $A\sim B$ iff $A$ and $B$ admit a common local orthocomplement $C$. The relation $\sim$ is an equivalence, and $\mathcal{E}(\mathcal{T})/\sim$ can be organized as an effect algebra whenever $\mathcal{T}$ satisfies the following property called algebraicity: For each quadruple of events $A,B,C,D$ such that $A\cap B=B\cap C=C\cap D=\emptyset$, we have    
    \begin{align}\label{eq:algebraicity}
        A\cup B\in\mathcal{T}\ \&\ B\cup C\in\mathcal{T}\ \&\ C\cup D\in\mathcal{T}\implies A\cup D\in\mathcal{T}.
    \end{align}
Note that we can organize a test space $\mathcal{T}$ as an $\epsilon$-simplicial set: For $n\geq 0$, $n$-simplices correspond to $n$-tuples of pairwise disjoint events $A_1,\ldots,A_n$ such that $A_1\cup\cdots\cup A_n\subseteq T$ for some test $T$. For $\epsilon$-edges, we take all tests. 
The algebraic property~\eqref{eq:algebraicity} then corresponds to the right lifting property w.r.t.:
    \begin{align}
      \begin{tikzcd}[row sep={3em,between origins}, column sep={3em,between origins}, ampersand replacement=\&,execute at end picture={
\foreach \Nombre in  {x,y,w,o,z}
  {\coordinate (\Nombre) at (\Nombre.center);}
\fill[opacity=0.3,pattern = vertical lines] 
      (x) -- (y)-- (z) -- (w)--(o)--cycle;}]
    |[alias=y]|y \&  \&|[alias=z]|z\arrow[ll,LA=1.5pt]\arrow[dd,LA=1.5pt]\arrow[ld,"C"{description}]\\
    \&|[alias=o]|o\arrow[lu,"B"{description}]\arrow[rd,"D"{description}]\&\\
    |[alias=x]|x\arrow[ru,"A"{description}]\arrow[uu,LA=1.5pt]\&\&|[alias=w]|w
\end{tikzcd} \ \ \ \hookrightarrow\ \ \ 
      \begin{tikzcd}[row sep={3em,between origins}, column sep={3em,between origins}, ampersand replacement=\&,execute at end picture={
\foreach \Nombre in  {x,y,w,o,z}
  {\coordinate (\Nombre) at (\Nombre.center);}
\fill[opacity=0.3,pattern = vertical lines] 
      (x) -- (y)-- (z) -- (w)--cycle;}]
    |[alias=y]|y \&  \&|[alias=z]|z\arrow[ll,LA=1.5pt]\arrow[dd,LA=1.5pt]\arrow[ld,"C"{description}]\\
    \&|[alias=o]|o\arrow[lu,"B"{description}]\arrow[rd,"D"{description}]\&\\
    |[alias=x]|x\arrow[rr,LA=1.5pt]\arrow[ru,"A"{description}]\arrow[uu,LA=1.5pt]\&\&|[alias=w]|w
\end{tikzcd}
\end{align}
\end{example}

Next, let us modify the nerve functor from Section~\ref{sec:3}  for Frobenius algebras as follows:
\begin{definition}\label{ex:SimFrob}
    Let $\mathcal{F}=(F;\mu,\eta,\delta,\epsilon)$ be a Frobenius algebra in $\Rel$. By $N(\mathcal{F})$ we denote an $\epsilon$-simplicial set, where the restriction of $N(\mathcal{F})$ to $\Delta$ is the nerve of the monoid $(F;\mu,\eta)$, $N(\mathcal{F})_\eps=\epsilon\subseteq F$, and the mapping $N(\mathcal{F})_\eps\rightarrow N(\mathcal{F})_1$ coincides with the inclusion $\eps\subseteq F=N(\mathcal{F})_1$.
\end{definition}
\begin{proposition}
    There is a fully faithful embedding $N\colon\RelFrob\hookrightarrow \eSSet$.
\end{proposition}
\begin{proof}
Let $\mathcal{F}_1, \mathcal{F}_2\in\RelFrob$. Recalling Proposition~\ref{prop:fullfaithful}, we see that the natural transformations in $\Hom_{\eSSet}(N(\mathcal{F}_1),N(\mathcal{F}_2))$ are in a bijection with the monoid homomorphisms $\mathcal{F}_1\rightarrow\mathcal{F}_2$ which preserve counit elements. The last homomorphisms are precisely the homomorphisms of the Frobenius algebra by Corollary~\ref{cor:frobHom}.
\end{proof}
For $n\geq 1$, denote by $\Sim^n$ an $\epsilon$-simplicial set that arises from $\Delta^n$ by marking the edge $\Delta^{\{0,n\}}$. In analogy to~\eqref{eq:faceAndHorn}, we denote by $\partial_i\Sim^n$ the corresponding face of $\Sim^n$, where the marked edge remains marked if it belongs to $\partial_i\Sim^n$ (i.e., $i\not=0,n$). Moreover, we set 
\begin{align}
\ELambda_i^n=\bigcup_{j\not=i}\partial_j\Sim^n.    
\end{align}
For Frobenius algebras, so-called \emph{$\epsilon$-horns} are important:
\begin{definition}
Let $n\geq 1$. By an $\epsilon$-horn we call an $\epsilon$-simplicial set of the form 
\begin{align}
    \ELambda^n_{0}&\subset \Sim^n \label{eq:domLift0}\text{\ or\ }\\
    \ELambda^n_{n}&\subset\Sim^n. \label{eq:domLiftn}
\end{align}
The set of all $\epsilon$-horns we denote by $\eHorn:=\{\ELambda^n_{0}\mid n\geq 1\}\cup\{\ELambda^n_{n}\mid n\geq 1\}.$
\end{definition}
It turns out that an $\epsilon$-simplicial set which arises from a Frobenius algebra admits a unique filling of all $\epsilon$-horns.
\begin{theorem}\label{thm:epsHorns}
    Let $\mathcal{F}=(F;\mu,\eta,\delta,\epsilon)$ be a Frobenius algebra in $\Rel$. Then the corresponding $\epsilon$-simplicial set $N(\mathcal{F})$ admits {unique} extension property w.r.t. inclusions~\eqref{eq:domLift0} and~\eqref{eq:domLiftn}, with $n\geq 1$.

\end{theorem}
\begin{proof}
Let $n=1$, $i=0,1$. The $\epsilon$-horn $\ELambda^n_i$ is isomorphic to $\Delta^0$. Hence $N(\mathcal{F})$ has the unique extension property w.r.t. $\ELambda^n_i\hookrightarrow\Sim^n$ if and only if every vertex in $N(\mathcal{F})$ is a source of a unique $\epsilon$-edge (the case $i=1$) and a target of a unique $\epsilon$-edge (the case $i=0$). This property follows from Lemma~\ref{lem:sourceAndTarget}. 

Let $n=2$. We will prove the case $i=0$ (the case $i=2$ is analogous). Assume an instance of the lifting problem in concern, that is, a morphism $\ELambda^2_0\hookrightarrow N(\mathcal{F})$. The morphism is given by an edge $a\in F$ and an $\epsilon$-edge $e\in\epsilon$ having the same source. We need to find an element $x\in F$, such that $\mu\colon (x,a)\dashmapsto e$. 
\begin{align}\label{defAlphaX}
        \begin{tikzcd}[row sep={1.6em,between origins}, column sep={2.2em,between origins}, ampersand replacement=\&]
    \&  2 \&
    \\[6pt]\&
    \&\\
    0
    \arrow[ruu, LA=1.5pt, "e"{description}]
    \arrow[rr,"a"{description}]
    \&\&
    1
\end{tikzcd}\ \ \hookrightarrow\ \ 
 \begin{tikzcd}[row sep={1.6em,between origins}, column sep={2.2em,between origins}, ampersand replacement=\&,execute at end picture={
\foreach \Nombre in  {A,B,C}
  {\coordinate (\Nombre) at (\Nombre.center);}
\fill[opacity=0.3,pattern = vertical lines] 
      (A) -- (B)-- (C) -- cycle; \foreach \Nombre in  {A,B,C}
  {\coordinate (\Nombre) at (\Nombre.center);}}]
    \&  |[alias=C]|2 \&
    \\[6pt]\&
    \&\\
    |[alias=A]|0
    \arrow[ruu,LA=1.5pt, "e"{description}]
    \arrow[rr,"a"{description}]
    \&\&
    |[alias=B]|1\arrow[uul,"x"{description}]
\end{tikzcd}
\end{align}
Let $x$ be the element $\hat{\beta}(a)$ given by Lemma~\ref{lem:alpha}, point (i). Then there is $f\in\epsilon$ such that $\mu\colon (a,x)\dashmapsto f$. Since $f$ has the same source as $e$, we obtain $e=f$ by Lemma~\ref{lem:alpha} (ii). Because $\hat{\beta}(a)$ is a unique element $x$ having this property, the solution~\eqref{defAlphaX} is indeed unique.

Let $n=3$. Again, we will prove only the case $i=0$. Assume a morphism 
$\phi\colon\ELambda^3_0\rightarrow N(\mathcal{F})$ which we want to extend as $\hat{\phi}\colon\Sim^3\rightarrow N(\mathcal{F})$. Let
$a,b,c,e,x,y\in F$ label edges of $\phi(\ELambda^3_0)$ as follows:
\begin{align}\label{sim:alpha33}
      \begin{tikzcd}[row sep={2.4em,between origins}, column sep={3.3em,between origins}, ampersand replacement=\&]
    \&
    {3}
    \&\\[10pt]\&
    {2}\arrow[u,"c"{description}]
    \&\\
    {0}
    \arrow[ruu,LA=1.5pt,"e"{description}, bend left=5]
    \arrow[rr, "a"{description}, bend right=5]
    \arrow[ru, "{x}"{description}]
    \&\&
    {1}\arrow[uul,"y"{description}, bend right=5]\arrow[lu, "b"'{description,name=f13}]
\end{tikzcd} 
\end{align}

The faces $\phi(\partial_i\Sim^3)$, for $i=1,2,3$, code (1) $\mu\colon (c,x)\dashmapsto e$, (2) $\mu\colon(y,a)\dashmapsto e$, and (3) $\mu\colon (b,a)\dashmapsto x$. Applying associativity to (1) and (3), we get an edge $z\in F$, such that (4) $\mu\colon (c,b)\dashmapsto z$ and (5) $\mu\colon (z,a)\dashmapsto e$. From (2) and (5) we deduce 
$$y=\hat{\beta}(a)=z.$$
Hence (4) gives us the desired face $\hat{\phi}(\partial_0\Delta^3)$ of~\eqref{sim:alpha33}. Due to Theorem~\ref{thm:2coskeletal} (2-coskeletality), we can define $\hat{\phi}$ on  whole $\Sim^3$. By the same theorem, we obtain also the uniqueness.

The remaining cases with $n\geq 4$ follow from the 2-cosketalitity property (Theorem~\ref{thm:2coskeletal}). Indeed, assume a morphism $\phi\colon \ELambda^n_0\rightarrow N(\mathcal{F})$. Since $\partial_0\Sim^n\cap \ELambda^n_0\cong \partial \Delta^{n-1}$, 
we can, by $2$-coskeletality, uniquely extend $\phi$ on the missing face $\partial_0\Sim^n$ ($n-1>2$). Applying $2$-coskeletality again, we uniquely extend $\phi$ onto the whole $\Sim^n$. We can proceed similarly in the case of $\ELambda^n_n$.
\end{proof}
The basic aim of this article is to translate the algebraic structure of Frobenius algebras into the combinatorics of simplices.
In this sense, we define the relations $\mu$, $\hat{\alpha}$, and $\hat{\beta}$ for a general $\epsilon$-simplicial set.
\begin{definition}\label{def:QuasiRel}
    Let $X$ be an $\epsilon$-simplicial set. We define the following relations on the set of edges $\mu\colon X_1\times X_1\dashrightarrow X_1$, $\hat{\alpha}\colon X_1\dashrightarrow X_1$, and $\hat{\beta}\colon X_1\dashrightarrow X_1$. For $a,b,c\in X_1$, we set $\mu\colon (b,a)\dashmapsto c$ if in $X$, there is the left-hand side $2$-simplex of~\eqref{defAlpha}. We will write $\hat{\alpha}\colon b\dashmapsto a$ and $\hat{\beta}\colon a\dashmapsto b$ if in $X$, there is the right-hand side $2$-simplex of~\eqref{defAlpha}.
    \begin{align}\label{defAlpha}
      \begin{tikzcd}[row sep={1.6em,between origins}, column sep={2.2em,between origins}, ampersand replacement=\&,execute at end picture={
\foreach \Nombre in  {A,B,C}
  {\coordinate (\Nombre) at (\Nombre.center);}
\fill[opacity=0.3,pattern = vertical lines] 
      (A) -- (B)-- (C) -- cycle; \foreach \Nombre in  {A,B,C}
  {\coordinate (\Nombre) at (\Nombre.center);}}]
    \&  |[alias=C]|2 \&
    \\[6pt]\&
    \&\\
    |[alias=A]|0
    \arrow[ruu,"c"{description}]
    \arrow[rr,"a"{description}]
    \&\&
    |[alias=B]|1\arrow[uul,"b"{description}]
\end{tikzcd}      \ \ \ \ 
 \begin{tikzcd}[row sep={1.6em,between origins}, column sep={2.2em,between origins}, ampersand replacement=\&,execute at end picture={
\foreach \Nombre in  {A,B,C}
  {\coordinate (\Nombre) at (\Nombre.center);}
\fill[opacity=0.3,pattern = vertical lines] 
      (A) -- (B)-- (C) -- cycle; \foreach \Nombre in  {A,B,C}
  {\coordinate (\Nombre) at (\Nombre.center);}}]
    \&  |[alias=C]|2 \&
    \\[6pt]\&
    \&\\
    |[alias=A]|0
    \arrow[ruu,"e"{description},LA=1.5pt]
    \arrow[rr,"a"{description}]
    \&\&
    |[alias=B]|1\arrow[uul,"b"{description}]
\end{tikzcd}
\end{align}
\end{definition}
\begin{lemma}
    For a Frobenius algebra $\mathcal{F}$, the relations $\mu$, $\hat{\alpha}$, and $\hat{\beta}$ coincide with those of Definition~\ref{def:QuasiRel} for $N(\mathcal{F})$. 
\end{lemma}
\begin{proof}
    It is immediate from the construction of $N(\mathcal{F})$.
\end{proof}
Let us describe some consequences of the filling of $\ELambda^n_i$, $i=0,n$, for a general $\epsilon$-set $X$. 
\begin{itemize}
    \item The case $n=1$ guarantees that each vertex is both a source of an $\epsilon$-edge and a target of an $\epsilon$-edge.
    \item The case $n=2$ enables us to find for a general edge $a$ an edge $\hat{\alpha}(a)$ ($\hat{\beta}(a)$, resp.) which mimics the rotation $\hat{\alpha}$ ($\hat{\beta}$, resp.) of Frobenius algebras.
    \item The case $n=3$ captures a sort of uniqueness of $\hat{\alpha}(a)$ ($\hat{\beta}(b)$, respectively) from the previous point. Assume an edge $a$ and an $\epsilon$-edge $e$ such as in~\eqref{defAlpha} and let there be two edges $b_1$ and $b_2$ both fitting in~\eqref{defAlpha} (left). Then we can define an $\epsilon$-horn $\phi\colon\ELambda^3_0\rightarrow X$ as follows:
    \begin{align}\label{sim:alpha3}
      \begin{tikzcd}[row sep={2.4em,between origins}, column sep={3.3em,between origins}, ampersand replacement=\&]
    \&
    {3}
    \&\\[10pt]\&
    {2}\arrow[u,"b_1"{description}]
    \&\\
    {0}
    \arrow[ruu,LA=1.5pt,"e"{description}, bend left=5]
    \arrow[rr, "a"{description}, bend right=5]
    \arrow[ru, "a"{description}]
    \&\&
    {1}\arrow[uul,"b_2"{description}, bend right=5]\arrow[lu, "t"'{description,name=f13}]
\end{tikzcd} 
\end{align}
Where the face $\phi(\partial_3\Sim^3)$ equals $s^1_1(a)$, hence $t$ is a degenerate edge. In the case $X=N(\mathcal{F})$, $t=t(a)$ is a unit element in $\eta$, and the face $\hat{\phi}(\partial_0\Sim^3)$, which we get from a filling, yields $\mu\colon (b_1,t)\dashmapsto b_2$, that is, $b_1=b_2$.
\end{itemize}
{\bf Notation:}
    For an edge $a$ of a general $\epsilon$-simplicial set, there may be none or many edges $b$ satisfying $\hat{\alpha}\colon a\dashmapsto b$. 
However, according to the above discussion, if $X$ admits fillings of $\epsilon$-horns, the relation $\hat{\alpha}$ in some sense evokes a bijection.
    In order to emphasize this, we will write $b=\hat{\alpha}(a)$ instead of $\hat{\alpha}\colon a\dashmapsto b$. In diagrams, a label $\hat{\alpha}(a)$ of an edge means any edge $b$ with the property $b=\hat{\alpha}(a)$. The same holds for $\hat{\beta}$.

Since we have defined $\mu$, $\hat{\alpha}$, and $\hat{\beta}$ for a general $\epsilon$-simplicial set $X$, we can also define $\delta$ by mimicking \eqref{eq:delta}. We have two options (use $\hat{\alpha}$ or $\hat{\beta}$); fortunately, both are equivalent if $X$ admits fillings of $\epsilon$ horns.
\begin{lemma}
Let $X$ be an $\epsilon$-simplicial set which has the extension property w.r.t. (\ref{eq:domLift0}--\ref{eq:domLiftn}), and let $a,b,c$ be edges of $X$.     In $X$, there is the left-hand side $2$-simplex in~\eqref{eq:des}, if and only if there is the right-hand side $2$-simplex.
    \begin{align}\label{eq:des}
        \begin{tikzcd}[row sep={1.6em,between origins}, column sep={2.2em,between origins}, ampersand replacement=\&]
    \&  2 \&
    \\[6pt]\&
    \&\\
    0
    \arrow[ruu, "a"{description}]
    \arrow[rr,"\hat{\alpha}(b)"{description}]
    \&\&
    1\arrow[uul,"c"{description}]
\end{tikzcd}\ \ \ \ 
 \begin{tikzcd}[row sep={1.6em,between origins}, column sep={2.2em,between origins}, ampersand replacement=\&]
    \&  2 \&
    \\[6pt]\&
    \&\\
    0
    \arrow[ruu, "b"{description}]
    \arrow[rr,"c"{description}]
    \&\&
    1\arrow[uul,"\hat{\beta}(a)"{description}]
\end{tikzcd}
    \end{align}
\end{lemma}
\begin{proof}
    Assume we have the left-hand side $2$-cell. This together with a $2$-simplex witnessing $\hat{\alpha}(b)$ gives us two faces of the following $3$-simplex $\phi\colon \partial_2\Sim^3\cup\partial_3\Sim^3\rightarrow X$:
    \begin{align}
          \begin{tikzcd}[row sep={2.4em,between origins}, column sep={3.3em,between origins}, ampersand replacement=\&]
    \&
    {3}
    \&\\[10pt]\&
    {2}
    \&\\
    {0}
    \arrow[ruu,LA=1.5pt, bend left=5]
    \arrow[rr,"\hat{\alpha}(b)"{description}  bend right=5]
    \arrow[ru, "{a}"{description}]
    \&\&
    {1}\arrow[uul,"{b}"{description}, bend right=5]\arrow[lu, "c"'{description,name=f13}]
\end{tikzcd}\ \ \hookrightarrow\ \ 
          \begin{tikzcd}[row sep={2.4em,between origins}, column sep={3.3em,between origins}, ampersand replacement=\&]
    \&
    {3}
    \&\\[10pt]\&
    {2}\arrow[u,"\hat{\beta}(a)"{description}]
    \&\\
    {0}
    \arrow[ruu,LA=1.5pt, bend left=5]
    \arrow[rr,"\hat{\alpha}(b)"{description}  bend right=5]
    \arrow[ru, "{a}"{description}]
    \&\&
    {1}\arrow[uul,"{b}"{description}, bend right=5]\arrow[lu, "c"'{description,name=f13}]
\end{tikzcd}
    \end{align}
    We can fill the left-hand side to the full $3$-simplex in the following two steps: First, we extend $\phi$ on the face $\Sim^{\{0,2,3\}}$ by applying the extension property over $\ELambda^2_0\hookrightarrow\Sim^2$. The three faces we have form the $\epsilon$-horn $\phi'\colon\ELambda^3_0\rightarrow X$, which we can fill. The last face that we have obtained is the desired right-hand side in \eqref{eq:des}.

    The opposite implication is analogous.
\end{proof}
\begin{definition}\label{def:deltaForSim}
    Let $X$ be an $\epsilon$-simplicial set and $a,b,c\in X_1$. We will write $\delta\colon (a,b)\dashmapsto c$ if there is a $3$-simplex in $X_3$ of the following form:
        \begin{align}\label{eq:deltaD}
          \begin{tikzcd}[row sep={2.4em,between origins}, column sep={3.3em,between origins}, ampersand replacement=\&,execute at end picture={
\foreach \Nombre in  {A,B,C}
  {\coordinate (\Nombre) at (\Nombre.center);}
    \fill[opacity=0.1] (A) -- (B) -- (D) -- cycle;
 \foreach \Nombre in  {A,B,C}
  {\coordinate (\Nombre) at (\Nombre.center);}}]
    \&
    |[alias=C]|3
    \&\\[10pt]\&
    {2}\arrow[u]
    \&\\
    |[alias=A]|0
    \arrow[ruu,LA=1.5pt, bend left=5]
    \arrow[rr,  bend right=5]
    \arrow[ru, "{a}"{description}]
    \&\&
    |[alias=B]|1\arrow[uul,"{b}"{description}, bend right=5]\arrow[lu, "c"'{description,name=f13}]
\end{tikzcd} 
\end{align}
\end{definition}
As in the case of Definition~\ref{def:QuasiRel}, if $X=N(\mathcal{F})$ for some Frobenius algebra $\mathcal{F}$, it is easy to prove that comultiplication $\delta$ of $\mathcal{F}$ coincides with that of Definition~\ref{def:deltaForSim}.

At this point, we can investigate a Frobenius identity for a general $\epsilon$-simplicial set $X$. The three parts of the Frobenius identity~\eqref{eq:frobId} correspond to the following three sub-$\epsilon$-simplicial sets of $\Sim^4$: 
    \begin{align}
       \Sim^{\{0,1,2,4\}} \cup \Delta^{\{0,2,3\}}&\subset \Sim^4 \label{i},\\
        \Sim^{\{0,1,3,4\}} \cup \Delta^{\{1,2,3\}}&\subset \Sim^4, \label{ii}\\
         \Sim^{\{0,2,3,4\}} \cup \Delta^{\{1,2,4\}}&\subset \Sim^4.\label{iii}
    \end{align}
See the following picture:
     \begin{gather*}\label{eq:Defdelta}
          \begin{tikzcd}[row sep={3.6em,between origins}, column sep={5em,between origins}, ampersand replacement=\&,execute at end picture={
\foreach \Nombre in  {A,B,C,E}
  {\coordinate (\Nombre) at (\Nombre.center);}
    \fill[opacity=0.1] (A) -- (B) -- (C) --(E) -- cycle; \foreach \Nombre in  {A,C,D}
  {\coordinate (\Nombre) at (\Nombre.center);}
\fill[pattern=north east lines,opacity=0.3] 
  (D) -- (A) -- (C) -- cycle;
}]
    \&
    |[alias=E]|4
    \&\&\\[10pt]\&
    |[alias=D]|3
    \&\&|[alias=C]|2\arrow[ll,"{a}"{description}]\arrow[llu]\\
    |[alias=A]|0\arrow[rrru]
    \arrow[ruu,LA=1.5pt, bend left=5]
    \arrow[rr, bend right=5]
    \arrow[ru, "{c}"{description}]
    \&\&
    |[alias=B]|1\arrow[uul,"{d}"{description}, bend right=5]\arrow[ru,"b"{description}]\&\\[-15pt]
    \&\eqref{i}\&\&
\end{tikzcd} 
          \begin{tikzcd}[row sep={3.6em,between origins}, column sep={5em,between origins}, ampersand replacement=\&,execute at end picture={
\foreach \Nombre in  {A,B,C}
  {\coordinate (\Nombre) at (\Nombre.center);}
    \fill[opacity=0.1] (A) -- (B) -- (C) -- cycle;
   \foreach \Nombre in  {D,B,C}
  {\coordinate (\Nombre) at (\Nombre.center);}
\fill[pattern=north east lines,pattern color=gray!50] 
  (D) -- (B) -- (C) -- cycle;
}]
    \&
    |[alias=C]|4
    \&\&\\[10pt]\&
    {3}\arrow[u]
    \&\&|[alias=B]|2\arrow[ll,"{a}"{description}]\arrow[llu]\\
    |[alias=A]|0\arrow[rrru]
    \arrow[ruu,LA=1.5pt, bend left=5]
    \arrow[ru, "{c}"{description}]
    \&\&
    |[alias=D]|1\arrow[uul,"{d}"{description}, bend right=5]\arrow[ru,"b"{description}]\&\\[-15pt]
    \&\eqref{iii}\&\&
\end{tikzcd} \end{gather*}
\begin{gather*}
         \begin{tikzcd}[row sep={3.6em,between origins}, column sep={5em,between origins}, ampersand replacement=\&, execute at end picture={
\foreach \Nombre in  {A,B,C}
  {\coordinate (\Nombre) at (\Nombre.center);}
    \fill[opacity=0.1] (A) -- (B) -- (C) -- cycle; 
  \foreach \Nombre in  {D,E,B}
  {\coordinate (\Nombre) at (\Nombre.center);}
\fill[pattern=north east lines,opacity=0.3] 
  (D) -- (E) -- (B) -- cycle;
}]
    \&|[alias=C]|4\&\&\\[10pt]\&
    |[alias=E]|3\arrow[u]
    \&\&|[alias=D]|2\arrow[ll,"{a}"{description}]\\
    |[alias=A]|0
    \arrow[ruu,LA=1.5pt, bend left=5]
    \arrow[rr,  bend right=5]
    \arrow[ru, "{c}"{description}]
    \&\&
    |[alias=B]|1\arrow[uul,"{d}"{description}, bend right=5]\arrow[lu]\arrow[ru,"b"{description}]\&\\[-15pt]
    \&\eqref{ii}\&\&
\end{tikzcd}
\end{gather*}
We will prove in Theorem~\ref{thm:ret}, that for an $\epsilon$-simplicial set $X$ which admits fillings of $\epsilon$-horns, the Frobenius identity (equivalence of (\ref{i}--\ref{iii})) holds whenever $X$ has the extension property w.r.t.: $$ \partial_1\Sim^4\cup\partial_3\Sim^4\hookrightarrow\Sim^4.$$
\section{Weakly saturated classes}
We aim to characterize (up to isomorphism) $\epsilon$-simplicial sets of the form $N(\mathcal{F})$, where $\mathcal{F}$ is a Frobenius algebra in $\Rel$. To achieve this in a convenient way, we need a supply of lifting problems we can solve (behind the generating ones). For these reasons, we introduce the concept of \emph{weakly saturated classes}. 
However, we will not deal with model structures in this article. That will be the aim in the future.
\begin{definition}[\cite{L}]\label{def:saturated}
Given a nonempty class of morphisms $\mathcal{I}$ in a cocomplete category $\mathcal{C}$, we say $\mathcal{I}$ is weakly saturated if it is closed under:
    \begin{enumerate}[(i)]
    \item compositions,
    \item transfinite compositions,
    \item pushouts,
    \item retracts.
\end{enumerate}
\end{definition}
As a consequence, a saturated class is also closed under coproducts and contains all isomorphisms. Recall a relation $\boxslash$ on a class of morphisms. It produces a closure operator on classes of morphisms $\mathcal{I}\mapsto \prescript{\boxslash}{}{(\mathcal{I}^\boxslash)}$.
For many important categories (i.e. presentable ones), the weakly saturated classes are exactly the ones closed under the above closure operator.
\begin{theorem}\label{thm:closure}
    Let $\mathcal{I}$ be a set of morphisms in $\eSSet$. Then $\prescript{\boxslash}{}{(\mathcal{I}^\boxslash)}$ is the least weakly saturated class containing $\mathcal{I}$.
\end{theorem}
\begin{proof}
    Since $\eSSet$ is a cocomplete presentable category, we can apply the small object argument in a standard way. For more details, see~\cite{MH}.
\end{proof}
    In this article, we will not need Theorem~\ref{thm:closure} in its full generality. We will only use that $\prescript{\boxslash}{}{(\mathcal{I}^\boxslash)}$ is closed under finite compositions, pushouts, and retracts, which one can prove by straightforward ``diagram chasing".
    
We will denote the closure of $\mathcal{I}$ in the sense of Theorem~\ref{thm:closure} as $\overline{\mathcal{I}}$.
Note that an $\epsilon$-simplicial set $X$ has the extension property w.r.t. all elements in some class of morphisms $\mathcal{I}$ iff the unique morphism $!_X\colon X\rightarrow \{\bullet\}$ to the terminal object satisfies $\mathcal{I}\ \boxslash\ !_X$. Hence 
\begin{align}\label{eq:sol}
    \mathcal{I}\boxslash X \Longleftrightarrow \mathcal{I}\ \boxslash\ !_X\ \Longleftrightarrow \overline{\mathcal{I}}\boxslash X.
\end{align}
\begin{lemma}\label{lem:pushOut}
    Let $X$ and $Y$ be $\epsilon$-simplicial sets and let $\mathcal{I}$ be a saturated class. If $i\colon X\cap Y\hookrightarrow Y$ belongs to $\mathcal{I}$, then $j\colon X\hookrightarrow X\cup Y$ belongs to $\mathcal{I}$ as well.
\end{lemma}
\begin{proof}
    Since colimits in $\eSSet$ are computed pointwise, we see that $j$ is a pushout of $i$. Therefore, $j$ belongs to $\mathcal{I}$ whenever $i$ does.
\end{proof}
As an example, let as show that $\partial_3\Sim^3\hookrightarrow\Sim^3$ belongs to $\overline{\eHorn}$.
\begin{equation*}
          \begin{tikzcd}[row sep={2.9em,between origins}, column sep={4em,between origins}, ampersand replacement=\&]
    \&
    {\ }
    \&\\[10pt]\&
    {2}
    \&\\
    {0}
    \arrow[rr, "a"{description}, bend right=5]
    \arrow[ru, "{c}"{description}]
    \&\&
    {1}\arrow[lu, "b"'{description,name=f13}]
\end{tikzcd} \xhookrightarrow{(1)}
          \begin{tikzcd}[row sep={2.9em,between origins}, column sep={4em,between origins}, ampersand replacement=\&]
    \&
    {3}
    \&\\[10pt]\&
    {2}\arrow[u,"\beta(c)"{description}]
    \&\\
    {0}
    \arrow[ruu,"{e}"{description},LA=1.5pt, bend left=5]
    \arrow[rr, "a"{description}, bend right=5]
    \arrow[ru, "{c}"{description}]
    \&\&
    {1}\arrow[lu, "b"'{description,name=f13}]
\end{tikzcd} \xhookrightarrow{(2)}
          \begin{tikzcd}[row sep={2.9em,between origins}, column sep={4em,between origins}, ampersand replacement=\&]
    \&
    {3}
    \&\\[10pt]\&
    {2}\arrow[u,"\beta(c)"{description}]
    \&\\
    {0}
    \arrow[ruu,"{e}"{description},LA=1.5pt, bend left=5]
    \arrow[rr, "a"{description}, bend right=5]
    \arrow[ru, "{c}"{description}]
    \&\&
    {1}\arrow[lu, "b"'{description,name=f13}]\arrow[luu,"\beta(a)"{description}]
\end{tikzcd}
    \end{equation*}
In the first step (1) we find an $\epsilon$-edge $e$ with a source $0$ and we fill the face $\Sim^{\{0,2,3\}}$ applying lifting over $\ELambda^2_0\hookrightarrow\Sim^2$. In the second step (2), we first fill the face $\Sim^{\{0,1,3\}}$, so we get $\ELambda^3_0$, which we can fill to the full $\Sim^3$. Note that we have applied Lemma~\ref{lem:pushOut} several times, e.g., in the second step (2), we have $X=\partial_1\Sim^3\cup\partial_3\Sim^3$, $Y=\partial_2\Sim^3$, and $X\cap Y=\{a\}\cup \{e\}$. Similarly, we have used the composition rule ((i) in Definition~\ref{def:saturated}).
\begin{theorem}\label{thm:lifting1}
    Let $n\in\mathbb{N}$ and $I\subseteq\{0,\ldots,n\}$ be such that $|I\cap \{0,n\}|=1$. Then
    \begin{align}\label{eq:genLift}
        \bigcup_{i\in I}\partial_{i}\Sim^n\hookrightarrow \Sim^n
    \end{align}
    belongs to $\overline{\eHorn}$.
\end{theorem}

\begin{proof}
    We use induction on $n$. In the case $n=1$, the embedding~\eqref{eq:genLift} coincides with~\eqref{eq:domLift0} or~\eqref{eq:domLiftn}. Let $n\geq 2$ and assume the claim holds for $n-1$. Without loss of generality, assume $I\cap \{0,n\}=\{0\}$; the case $I\cap \{0,n\}=\{n\}$ is analogous.

    If $I=\{0,\ldots,n-1\}$, then \eqref{eq:genLift} is an $\epsilon$-horn. Otherwise, there is $j\in\{1,\dots, n-1\}$ such that $j\notin I$. We claim that the following embedding belongs to $\overline{\eHorn}$
    \begin{equation}
         \bigcup_{i\in I}\partial_{i}\Sim^n\hookrightarrow  \partial_j\Sim^n\cup\bigcup_{i\in I}\partial_{i}\Sim^n. \label{eq:embi}
    \end{equation}
    Indeed, $\partial_j\Sim^n\cap \bigcup_{i\in I}\partial_{i}\Sim^n\hookrightarrow \partial_j\Sim^n$ is isomorphic to $\bigcup_{i\in\Tilde{I}}\partial_i\Sim^{n-1}\hookrightarrow \Sim^{n-1}$, where $\Tilde{I}=\{i \mid i\in I, i<j\}\cup\{i-1\mid i\in I, i> j\}$. The latter embedding belongs to $\overline{\eHorn}$ by the induction hypothesis, hence \eqref{eq:embi} belongs to $\overline{\eHorn}$ by Lemma~\ref{lem:pushOut}. Repeating this for each $j\in\{1,\ldots,n-1\}$, $j\notin I$, we achieve the case $I=\{0,\ldots,n-1\}$.
\end{proof}

Note that the example before Theorem~\ref{thm:lifting1} describes the case of $I=\{3\}\subseteq\{0,1,2,3\}$. More generally, assume any $n$-simplex $x\colon\partial_{n+1}\Sim^{n+1}\cong\Delta^n\rightarrow X$, where $X$ admits fillings of $\epsilon$-horns. Then $x$ admits an extension $\hat{x}\colon\Sim^{n+1}\rightarrow X$ by Theorem~\ref{thm:lifting1}. The restriction of $\hat{x}$ onto $\partial_0\Sim^{n+1}\cong \Delta^n$: 

$$\hat{\alpha}(x)\colon\partial_0\Sim^{n+1}\cong\Delta^n\rightarrow X$$
may be interpreted as a rotation of $x$. In this way, one can simulate the $\mathbb{Z}$-action on simplices discussed in~\cite{MZ} using our framework.

Recall that associativity is a consequence of a Frobenius identity (Remark~\ref{rem:Asoc}). 
The following lemma provides a subtle comparison of the strength of the associativity rule and the Frobenius rule modeled with simplicial combinatorics.\begin{lemma}\label{lem:subtleComparison}
    Denote the following four sets of morphisms in $\eSSet$:
    \begin{enumerate}[(i)]
    \item $I_1=\eHorn\cup\{\partial_1\Sim^4\cup\partial_3\Sim^4\hookrightarrow\Sim^4\}$,
        \item $I_2=\eHorn\cup\{\partial_2\Sim^4\cup\Delta^{\{1,2,3\}}\hookrightarrow\Sim^4\}$,
        \item $I_3=\eHorn\cup\{\partial_1\Delta^3\cup\partial_3\Delta^3\hookrightarrow \Delta^3\}$,
        \item $I_4=\eHorn\cup\{\partial_0\Delta^3\cup\partial_2\Delta^3\hookrightarrow \Delta^3\}$.
    \end{enumerate}
    Then $\overline{I}_1\cup\overline{I}_2\subseteq \overline{I}_3\cap\overline{I}_4$.
\end{lemma}
\begin{proof}
We only need to prove that both embeddings $\phi_1\colon\partial_1\Sim^4\cup\partial_3\Sim^4\hookrightarrow\Sim^4$ and $\phi_2\colon \partial_2\Sim^4\cup\Delta^{\{1,2,3\}}\hookrightarrow\Sim^4$ belong to $\overline{I}_3$ and $\overline{I}_4$.

Let us begin with the proof of $\phi_1\in\overline{I}_4$. Consider the following diagram of inclusions where the square is a pushout:
\begin{equation}\label{diag:asoVsFrob}
    \begin{tikzcd}
        {\partial_0\Sim^4\cap (\partial_1\Sim^4\cup\partial_3\Sim^4)}\arrow[r,hook]\arrow[d,hook]&{\partial_0\Sim^4}\arrow[d,hook]&\\
        {\partial_1\Sim^4\cup\partial_3\Sim^4}\arrow[r,hook]&{\partial_0\Sim^4\cup\partial_1\Sim^4\cup\partial_3\Sim^4}\arrow[r,hook]&\Sim^4
    \end{tikzcd}
\end{equation}
    The upper inclusion equals $\Delta^{\{2,3,4\}}\cup\Delta^{\{1,2,4\}}\subset \Delta^{\{1,2,3,4\}}$ which is isomorphic to $\partial_0\Delta^3\cup\partial_2\Delta^3\hookrightarrow \Delta^3\in I_4$. Using Lemma~\ref{lem:pushOut} and Theorem~\ref{thm:lifting1}, we deduce $\phi_1$ is a composition of two morphisms in $\overline{I}_4$, hence itself belongs to $\overline{I}_4$.  

We can prove $\phi_1\in\overline{I}_3$ in a similar way. In~\eqref{diag:asoVsFrob}, we only replace $\partial_0\Sim^4$ with $\partial_4\Sim^4$ and observe that the upper inclusion $\Delta^{\{0,2,3\}}\cup\Delta^{\{0,1,2\}}\subset \Delta^{\{0,1,2,3\}}$ is isomorphic to $\partial_1\Delta^3\cup\partial_3\Delta^3\hookrightarrow \Delta^3\in I_4$.

Next we prove $\phi_2\in\overline{I}_3$. The argument is very similar to the one above. However, this time we use the following diagram. 
    \begin{equation}\label{diag:asoVsFrob2}
    \begin{tikzcd}
        {\partial_0\Sim^4\cap (\partial_2\Sim^4\cup\Delta^{\{1,2,3\}})}\arrow[r,hook]\arrow[d,hook]&{\partial_0\Sim^4}\arrow[d,hook]&\\
        {\partial_2\Sim^4\cup\Delta^{\{1,2,3\}}}\arrow[r,hook]&{\partial_0\Sim^4\cup\partial_2\Sim^4}\arrow[r,hook]&\Sim^4
    \end{tikzcd}
\end{equation}
In \eqref{diag:asoVsFrob2}, the upper inclusion $\Delta^{\{1,3,4\}}\cup\Delta^{\{1,2,3\}}\subset\Delta^4$ is isomorphic to $\partial_1\Delta^3\cup\partial_3\Delta^3\subset\Delta^3$.

The last case $\phi_2\in\overline{I}_3$ we can deduce using a diagram similar to~\eqref{diag:asoVsFrob2} where we replace $\partial_0\Sim^4$ with $\partial_4\Sim^4$. 
\end{proof}
\begin{lemma}\label{lem:equiv}
    Let $X$ be an $\epsilon$-simplicial set and $I_1,\ldots,I_4$ be as in Lemma~\ref{lem:subtleComparison}. We have
    \begin{align}
        I_1\boxslash X\Leftrightarrow I_2\boxslash X\Leftrightarrow  I_3\boxslash X\Leftrightarrow I_4\boxslash X.
    \end{align}
\end{lemma}
\begin{proof}
Due to Lemma~\ref{lem:subtleComparison}, we only need to prove $I_i \boxslash X\implies I_j\boxslash X$ for $i=1,2$ and $j=3,4$. All the four cases have an analogous proof which we demonstrate on the following case: Assume $I_1\boxslash X$, we will prove $I_4\boxslash X$. Consider the diagram~\eqref{diag:asoVsFrob} from the previous lemma. In the proof of Lemma~\ref{lem:subtleComparison}, we observed that the upper inclusion, let us denote it by $m$, is isomorphic to $\partial_0\Delta^3\cup\partial_2\Delta^3\hookrightarrow\Delta^3$. Hence, it is enough to show $m\boxslash X$. We claim it follows from the fact that the left-hand vertical inclusion in~\eqref{diag:asoVsFrob}, let us denote it by $u$, satisfies 
\begin{equation}\label{eq:u}
    (u\colon\partial_0\Sim^4\cap (\partial_1\Sim^4\cup\partial_3\Sim^4)\hookrightarrow \partial_1\Sim^4\cup\partial_3\Sim^4 )\boxslash X.
\end{equation}
Indeed, if $u\boxslash X$, then any $\theta\colon \partial_0\Sim^4\cap (\partial_1\Sim^4\cup\partial_3\Sim^4)\rightarrow X$ extends by the assumptions to $\hat{\theta}\colon\Sim^4\rightarrow X$. The desired extension arises by restriction of $\hat{\theta}$ to $\partial_0\Sim^4$.

In order to prove~\eqref{eq:u}, we consider the following refinement of $u$:
\begin{align*}
       \partial_0\Sim^4\cap (\partial_1\Sim^4\cup\partial_3\Sim^4)&=\Delta^{\{2,3,4\}}\cup\Delta^{\{1,2,4\}}\subset \Sim^{\{0,2,3,4\}}\cup\Delta^{\{1,2,4\}}\\
       &\subset \Sim^{\{0,2,3,4\}}\cup\Sim^{\{0,1,2,4\}}=\partial_1\Sim^4\cup\partial_3\Sim^4.
\end{align*}
 Both involved inclusions belong to $\overline{\eHorn}$ by Theorem~\ref{thm:lifting1} and Lemma~\ref{lem:pushOut}. This finishes the proof of $I_1 \boxslash X\implies I_4\boxslash X$. 

 One can prove the other implication analogously. In all the cases, it is enough to consider the corresponding variant of diagram~\eqref{diag:asoVsFrob} and show that the left-hand vertical inclusion belongs to $\overline{\eHorn}$.
 \end{proof}
The following theorem says that under the assumption of fillings of all $\epsilon$-horns, the lifting problem $\partial_1\Sim^4\cup\partial_3\Sim^4\subseteq\Sim^4$ faithfully represents 
the two lifting problems given by~\eqref{i} and~\eqref{iii} (which correspond to two sides of the Frobenius identity).
 \begin{theorem}\label{thm:ret}
    The following three sets of morphisms in $\eSSet$ generate the same weakly saturated class: 
    \begin{align}
        J_1&=\eHorn \cup\{j_1\colon \Sim^{\{0,1,2,4\}} \cup \Delta^{\{0,2,3\}}\subset \Sim^4 \},\\
        J_2&=\eHorn \cup\{j_2\colon\Sim^{\{0,2,3,4\}} \cup \Delta^{\{1,2,3\}}\subset \Sim^4 \},\\
        J_3&=\eHorn \cup\{j_3\colon\Sim^{\{0,1,2,4\}} \cup \Sim^{\{0,2,3,4\}}\subset \Sim^4 \}.
    \end{align}
\end{theorem}
\begin{proof}
   Let us first prove $\overline{J}_1=\overline{J}_3$. The inclusion $J_1\subseteq \overline{J}_3$ is easy since  $j_1$ splits through $j_3$:
    \begin{equation}
         \Sim^{\{0,1,2,4\} }\cup\Delta^{\{0,2,3\}}\subset \Sim^{\{0,1,2,4\}} \cup \Sim^{\{0,2,3,4\}}\subset \Sim^4,
    \end{equation}
    where the first inclusion belongs to $\overline{\eHorn}$ as it is a pushout of $\Delta^{\{0,2,3\}}\cup\Sim^{\{0,1,4\}}\subset \Sim^{\{0,2,3,4\}}$. To prove reverse inclusion, we show that $j'_3\colon \Sim^{\{0,1,2,5\}}\cup \Sim^{\{0,2,3,5\}}\subset \Sim^{\{0,1,2,3,5\}}$, which is obviously isomorphic to $j_3$, belongs to $\overline{J}_1$. 
    First, we verify the following composition of inclusions belongs to $\overline{J}_1$:
    \begin{align*}
        \Sim^{\{0,1,2,5\}}\cup \Sim^{\{0,2,3,5\}}&\overset{\scriptscriptstyle (1)}{\subset} \Sim^{\{0,1,2,5\}}\cup\Sim^{\{0,1,2,4\}} \cup \Sim^{\{0,2,3,5\}}\\
        &\overset{\scriptscriptstyle (2)}{\subset}\Sim^{\{0,1,2,5\}}\cup\Sim^{\{0,1,2,3,4\}}\cup\Sim^{\{0,2,3,5\}}\\
        &\overset{\scriptscriptstyle (3)}{\subset} \Sim^{\{0,1,2,4,5\}}\cup\Sim^{\{0,1,2,3,4\}}\cup\Sim^{\{0,2,3,5\}}\\
        &\overset{\scriptscriptstyle (4)}{\subset} \Sim^{\{0,1,2,4,5\}}\cup\Sim^{\{0,1,2,3,4\}}\cup\Sim^{\{0,2,3,4,5\}}\\
        &\overset{\scriptscriptstyle (5)}{\subset}\Sim^{\{0,1,2,3,4,5\}}\cup\Sim^{\{0,4\}}=\Sim^5\cup\Sim^{\{0,4\}}.
    \end{align*}
    The first inclusion (1) is a pushout of $\Delta^{\{0,1,2\}}\subset\Sim^{\{0,1,2,4\}} \in\overline{\eHorn}$. The second inclusion (2) is a pushout of $\Sim^{\{0,1,2,4\}}\cup\Delta^{\{0,2,3\}}\subset\Sim^{\{0,1,2,3,4\}}$ which equals $j_1$.  The third inclusion (3) is a pushout of $\Sim^{\{0,1,2,5\}}\cup\Delta^{\{0,1,2,4\}}\subset\Sim^{\{0,1,2,4,5\}}$ which belongs to $\overline{\eHorn}$ by Lemma~\ref{thm:lifting1}. The fourth inclusion (4) is a pushout of $\Sim^{\{0,2,4,5\}}\cup \Sim^{\{0,2,3,4\}}\cup \Sim^{\{0,2,3,5\}}\subset\Sim^{\{0,2,3,4,5\}}$ which belongs to $\overline{\eHorn}$ by Lemma~\ref{thm:lifting1}. The last inclusion (5) is a pushout of $\partial_3\Sim^5\cup\partial_5\Sim^5\cup\partial_1\Sim^5\subset\Sim^5$.

    The inclusion $j'_3$ in concern is a retract of the one above:
    \begin{equation*}\label{diag:retract}
    \begin{tikzcd}
        {\Sim^{\{0,1,2,5\}}\cup \Sim^{\{0,2,3,5\}}}\arrow[r,"="]\arrow[d,hook,"j'_3"]&{\Sim^{\{0,1,2,5\}}\cup \Sim^{\{0,2,3,5\}}}\arrow[d,hook]\arrow[r,"="]&{\Sim^{\{0,1,2,5\}}\cup \Sim^{\{0,2,3,5\}}}\arrow[d,hook,"j'_3"]\\
        {\Sim^{\{0,1,2,3,5\}}}\arrow[r,hook]&{\Sim^{\{0,1,2,3,4,5\}}\cup\Sim^{\{0,4\}}}\arrow[r,two heads,"\pi"]&{\Sim^{\{0,1,2,3,5\}}}
    \end{tikzcd}
\end{equation*}
The morphism $\pi$ is the unique one which acts on vertices as follows: $4\mapsto 5$ and for $i\not= 4$, $i\mapsto i$ (see that the $\epsilon$-edge $\Sim^{\{0,4\}}$ goes to the $\epsilon$-edge $\Sim^{\{0,5\}}$).

The case $\overline{J}_2=\overline{J}_3$ is established analogously. Note that if we reverse the orientation of simplices, i.e., we permute indices in each $\Delta^n$ as $i\mapsto n-i$, then the generators of $J_1$ translate to the generators of $J_2$ and vice versa. While the generating set of $J_3$ is invariant under reorientation.
\end{proof}

\section{Characterization of Frobenius algebras}
In this section, we provide a characterization of the image of $N\colon\RelFrob\rightarrow\eSSet$ by means of lifting properties.
\begin{theorem}\label{thm:frobenius}
    Let $X$ be an $\epsilon$-simplicial set which has the extension property w.r.t. all $\epsilon$-horns. 
    The relation $\mu$ and $\delta$ from Definition~\ref{def:QuasiRel} and~\ref{def:deltaForSim} satisfy the Frobenius identity~\eqref{eq:frobId} whenever $X$ has the extension property w.r.t.
    \begin{align}\label{eq:frobLift}
        \partial_1\Sim^4\cup\partial_3\Sim^4\hookrightarrow\Sim^4.
    \end{align}
\end{theorem}
\begin{proof}
Assume that $X$ has the extension property w.r.t.~\eqref{eq:frobLift}. As discussed at the end of Section~\ref{sec:eHorn}, we need to prove that for each quadruple $a,b,c,d\in X_1$,  whenever $X$ has an $\epsilon$-simplicial subset of either of cases (\ref{i}--\ref{iii}), then all of them occur in $X$. 
We show that in each of the three cases we can fill the whole $4$-simplex. Cases (i) and (iii) follow from Theorem~\ref{thm:ret}. 
We can handle case (ii) using Lemma~\ref{lem:equiv}. It follows from the equivalence $I_1\boxslash X\Leftrightarrow I_2\boxslash X$.
\end{proof}

The following result is analogous to the one characterizing quasi-categories which arise as a result of the nerve functor $N\colon\mathrm{Cat}\rightarrow \SSet$.
\begin{theorem}\label{thm:char}
An $\epsilon$-simplicial set $X$ is of the form $X\cong N(\mathcal{F})$ for some Frobenius algebra $\mathcal{F}$ in $\Rel$ if and only if $X$ has the unique extension property w.r.t.    \begin{enumerate}[(i)]
        \item $\ELambda^n_0\subset \Sim^n$ and $\ELambda^n_n\subset \Sim^n$ for $n\geq 1$ ($\epsilon$-horns),
        \item $\partial\Delta^n\subset \Delta^n$ for $n\geq 3$ ($2$-coskeletality),
    \end{enumerate}
    and $X$ has the extension property w.r.t.
    \begin{enumerate}
        \item[(iii)]     $\partial_1\Sim^4\cup\partial_3\Sim^4\hookrightarrow\Sim^4$ (Frobenius identity).
    \end{enumerate}
\end{theorem}
\begin{proof}
    The right-to-left implication is easy. Condition (i) follows from Theorem~\ref{thm:epsHorns}, $2$-coskeletality (ii) holds by Theorem~\ref{thm:2coskeletal}. Since each Frobenius algebra $\mathcal{F}$ is associative, we have $(\partial_1\Delta^3\cup\partial_3\Delta^3\hookrightarrow \Delta^3)\boxslash N(\mathcal{F})$ (see Example~\ref{ex:assoc}). Thus, the condition (iii) follows from Lemma~\ref{lem:equiv}. 

    Assume $X$ satisfies the lifting properties (i--iii). We will find a Frobenius algebra $\mathcal{F}=(X_1;\mu,\eta,\delta,\mu)$ on the set of edges of $X$ such that $N(\mathcal{F})\cong X$. We define $\mu$ and $\delta$ according to Definition~\ref{def:QuasiRel} and Definition~\ref{def:deltaForSim}. Next, let $\eta$ contain all degenerate edges and $\epsilon$ contain all $\epsilon$-edges. By Theorem~\ref{thm:frobenius}, $\mathcal{F}$ satisfies the Frobenius identity. 
Hence, it remains to prove the unit and counit rules (we obtain associativity and coassociativity as a consequence; see Remark~\ref{rem:Asoc}).    
    First, observe that applying the assumptions to $\ELambda^1_i\subset \Sim^1$, for $i=0,1$, we get there is a unique $\epsilon$-edge with a given source (target, resp.) and each $\epsilon$-edge is witnessed by a unique element in $X_\epsilon$. 
Similarly, we observe that the relations $\hat{\alpha}$ and $\hat{\beta}$ established in Section~\ref{sec:eHorn} are (due to uniqueness) bijections $X_1\rightarrow X_1$.

    Let us prove that $(X_1;\mu,\eta)$ satisfies the unit rule (Definition~\ref{def:monoid}, (ii)). 
Take any $a\in X_1$. Its target $t(a):=s^0_0\circ d^1_0(a)$ and source $s(a):=s^0_0\circ d^1_1(a)$ are degenerate edges (hence elements of $\eta$) and clearly fit in $\mu\colon (t(a),a)\dashmapsto a$ and $\mu\colon(a,s(a))\dashmapsto a$, respectively.
    
Next, assume $a,b\in X_1$ and let $r\in s^0_0(X_0)$ satisfy $\mu\colon(r,b)\dashmapsto a$ ($\mu\colon(b,r)\dashmapsto a$, respectively). We will prove $b=a$.
Note that $r=s_0^0\circ d^1_0(r)=s_0^0\circ d^1_0(a)=t(a)$ ($r=s(a)$, respectively).

Let $e_1$ ($e_2$, resp.) be the unique $\epsilon$-edge with $t(e_1)=t(a)$ ($s(e_2)=s(a)$, resp.). We have the following $\epsilon$-horn $\ELambda^3_3$ ($\ELambda^3_0$, resp.) in $X$:
\begin{align}\label{diag:3simWit}
\begin{tikzcd}[row sep={2.4em,between origins}, column sep={3.3em,between origins}, ampersand replacement=\&]
    \&
    {3}
    \&\\[10pt]\&
    {2}\arrow[u,"t(a)"{description}]
    \&\\
    {0}
    \arrow[ruu,LA=1.5pt,"{e_1}"{description}, bend left=5]
    \arrow[rr,"\hat{\alpha}(a)"{description}  bend right=5]
    \arrow[ru, LA=1.5pt,"{e_1}"{description}]
    \&\&
    {1}\arrow[uul,"{a}"{description}, bend right=5]\arrow[lu, "b"'{description}]
\end{tikzcd}\ \ \ \text{respective}\ \ \ 
          \begin{tikzcd}[row sep={2.4em,between origins}, column sep={3.3em,between origins}, ampersand replacement=\&]
    \&
    {3}
    \&\\[10pt]\&
    {2}\arrow[u,"\hat{\beta}(a)"{description}]
    \&\\
    {0}
    \arrow[ruu,LA=1.5pt, bend left=5,"e_2"{description}]
    \arrow[rr,"s(a)"{description}  bend right=5]
    \arrow[ru,"a"{description}]
    \&\&
    {1}\arrow[uul,LA=1.5pt,"e_2"{description}, bend right=5]\arrow[lu, "b"'{description,name=f13}]
\end{tikzcd}
    \end{align}
By the assumptions, horns in~\eqref{diag:3simWit} admit fillings. The so obtained face $\partial_3\Delta^3$ ($\partial_0\Delta^3$, resp.) witnesses $\hat{\alpha}(b)=\hat{\alpha}(a)$ ($\hat{\beta}(b)=\hat{\beta}(a)$, resp.). However, we have already observed $\hat{\alpha}$ ($\hat{\beta}$, resp.) is bijective, this gives us $a=b$.

Finally, we verify the counit rule. First, assume any $a\in X_1$ and let $e_1$ ($e_2$, resp.) be the unique $\epsilon$-edge satisfying $s(a)=s(e_1)$ ($t(a)=t(e_2)$, resp.). Using the assumed lifting properties, it is easy to show that in $X$, we have the following full (degenerate) $3$-simplices:
 \begin{align}\label{eq:witOfCounit}
          \begin{tikzcd}[row sep={2.4em,between origins}, column sep={3.3em,between origins}, ampersand replacement=\&]
    \&
    {3}
    \&\\[10pt]\&
    {2}\arrow[u,"t(a)"{description}]
    \&\\
    {0}
    \arrow[ruu,LA=1.5pt,"{e_1}"{description}, bend left=5]
    \arrow[rr,"\hat{\alpha}(a)"{description}  bend right=5]
    \arrow[ru, LA=1.5pt,"{e_1}"{description}]
    \&\&
    {1}\arrow[uul,"{a}"{description}, bend right=5]\arrow[lu, "a"'{description}]
\end{tikzcd}\ \ \ \text{respective}\ \ \ 
          \begin{tikzcd}[row sep={2.4em,between origins}, column sep={3.3em,between origins}, ampersand replacement=\&]
    \&
    {3}
    \&\\[10pt]\&
    {2}\arrow[u,"\hat{\beta}(a)"{description}]
    \&\\
    {0}
    \arrow[ruu,LA=1.5pt, bend left=5,"e_2"{description}]
    \arrow[rr,"s(a)"{description}  bend right=5]
    \arrow[ru,"a"{description}]
    \&\&
    {1}\arrow[uul,LA=1.5pt,"e_2"{description}, bend right=5]\arrow[lu, "a"'{description,name=f13}]
\end{tikzcd}
    \end{align}
Simplices~\eqref{eq:witOfCounit} witness $\delta\colon (e_1,a)\dashmapsto a$ ($\delta\colon (a,e_2)\dashmapsto a$, resp.). Next assume $a,b\in X_1$ such that there is an $\epsilon$-edge $e_1$ ($e_2$, resp.) so that $\delta\colon (e_1,a)\dashmapsto b$ ($\delta\colon (a,e_2)\dashmapsto b$). This must be witnessed by a $3$-simplex of form~\eqref{diag:3simWit} left (right, resp.). Therefore, $a=b$ by the already verified unit rule.  

It remains to show $X\cong N(\mathcal{F})$. 
As both $\epsilon$-simplicial sets are $2$-coskeletal, it is enough to describe the isomorphism on simplices of dimension up to $2$.
By the very construction of $\mathcal{F}$, there are isomorphisms $X_0\cong N(F)_0$, $X_1\cong N(F)_1$, and $X_2\cong N(F)_2$ which are compatible with degeneracy and face maps. Moreover, since each $\epsilon$-edge in $X$ has a unique witness in $X_\epsilon$, there is an isomorphism $N(\mathcal{F})_\epsilon=\epsilon\cong X_\epsilon$.
\end{proof}

    In Theorem~\ref{thm:char}, we may leave the uniqueness condition in (ii). Assume that we have two $n$-simplices $x,y\colon \Delta^n\rightarrow X$, $n\geq 3$, which coincide at the boundary $\partial\Delta^n$. Using Theorem~\ref{thm:lifting1}, we may extend $x$ to $\hat{x}\colon\Sim^{n+1}\rightarrow X$, so that $x$ coincides with the restriction of $\hat{x}$ on $\partial_{n+1}\Sim^{n+1}$. Denote by $z\colon\ELambda^{n+1}_{n+1}\rightarrow X$ the restriction of $\hat{x}$. Now both $z\cup x$ and $z\cup y$ define a mapping $\partial\Sim^{n+1}\rightarrow X$. Both of these extend to whole $\Sim^{n+1}\rightarrow X$ by (ii). However, the two obtained simplices are extensions of the $\epsilon$-horn $z$. So $x=y$ by the uniqueness condition in (i).
\section{Future work}
The main aim is to define a convenient model structure on $\epsilon$-simplicial sets, which would represent Frobenius algebras as fibrant objects. Such a model structure would provide a robust framework for studying higher Frobenius algebras.

The nerve functor $N\colon\RelFrob\rightarrow \eSSet$ admits a left adjoint. In the case of categories, we have an explicit description of the left adjoint on fibrant objects, given by the Boardman–Vogt construction. The author intends to describe an analogous construction for Frobenius algebras which, in particular, generalizes the construction of an effect algebra from a test space (the Foulis–Bennett construction).

For two groupoids $G$ and $H$, it is well-known that $\mathrm{Hom}(G,H)$ with natural transformations can be organized as a groupoid. Similarly, for two effect algebras $E$ and $F$, the functor complex $\mathrm{Fun}(N(E), N(F))$ is isomorphic to the nerve of a Frobenius algebra. This observation raises a natural question: does this isomorphism hold more generally for arbitrary Frobenius algebras in $\Rel$? An affirmative result would enable $\RelFrob$ to be treated as a category enriched over itself and obtain a new structural understanding of Frobenius algebras.

Furthermore, it is desirable to investigate which lifting properties are satisfied by the functor complex $\mathrm{Fun}(X,Y)$ for convenient $\epsilon$-simplicial sets $X,Y$. 
The aim is to internalize the lifting calculus and prove results analogous to those of A. Quillen on quasi-categories.
\section*{Declarations}
\subsection*{Acknowledgments:} The author acknowledges the support by the Czech Science Foundation (GAČR): project
24-14386L.
\subsection*{Conflict of nterests:}
The author declares that he has no conflict of interest.

\bibliography{bibliography}

\begin{thebibliography}{CKM22}

\bibitem[CKM22]{CKM}
I.~Contreras, M.~Keller, and R.~A. Mehta.
\newblock Frobenius objects in the category of spans.
\newblock {\em Reviews in Mathematical Physics}, 34(10), 2022.

\bibitem[CMS25]{CMS}
I.~Contreras, R.~A. Mehta, and W.~H. Stern.
\newblock Frobenius and commutative pseudomonoids in the bicategory of spans.
\newblock {\em Journal of Geometry and Physics}, 207, 2025.

\bibitem[Fou89]{F}
D.~J. Foulis.
\newblock Coupled physical systems.
\newblock {\em Foundations of Physics}, 19:905--922, 1989.

\bibitem[FPR83]{FPR}
D.~Foulis, C.~Piron, and C.~Randall.
\newblock Realism, operationalism, and quantum mechanics.
\newblock {\em Found. Phys.}, 13:813--841, 1983.

\bibitem[HCC13]{CCH}
C.~Heunen, I.~Contreras, and A.~S. Cattaneo.
\newblock Relative frobenius algebras are groupoids.
\newblock {\em Journal of Pure and Applied Algebra}, 217(1):114--124, 2013.

\bibitem[Hov98]{MH}
M.~Hovey.
\newblock {\em Model Categories}.
\newblock American Mathematical Society, Michigan, 1998.

\bibitem[HV20]{HV}
C.~Heunen and J.~Vicary.
\newblock {\em Categories for quantum theory}.
\newblock Oxford University Press, Oxford, 2020.

\bibitem[Jen18]{Je}
G.~Jen\v{c}a.
\newblock Effect algebras as presheaves on finite boolean algebras.
\newblock {\em Order}, 35:525–540, 2018.

\bibitem[Koc03]{JK}
J.~Kock.
\newblock {\em Frobenius algebras and 2D topological quantum field theories}.
\newblock Cambridge University Press, Cambridge, 2003.

\bibitem[Lur09]{L}
J.~Lurie.
\newblock {\em Higher Topos Theory}.
\newblock Princeton University Press, Princeton, 2009.

\bibitem[MZ20]{MZ}
R.~A. Metha and R.~Zhang.
\newblock Frobenius objects in the category of relations.
\newblock {\em Lett Math Phys}, 110:1941–1959, 2020.

\bibitem[PS16]{PS}
D.~Pavlovic and P.~M. Seidel.
\newblock (modular) effect algebras are equivalent to (frobenius) antispecial algebras, 2016.

\bibitem[Rou16]{R}
F.~Roumen.
\newblock Cohomology of effect algebras, 2016.

\bibitem[SU18]{SU}
S.~Staton and S.~Uijlen.
\newblock Effect algebras, presheaves, non-locality and contextuality.
\newblock {\em Information and Computation}, 261:336--354, 2018.

\end{thebibliography}
\end{document}